\documentclass[leqno,12pt]{article} %leqnoで数式番号を左側へ配置
 \usepackage{amsmath, amssymb}
 \usepackage{amsthm} %AMSのTheorem環境を導入
\numberwithin{equation}{section}
\theoremstyle{plain} % イタリック体
\newtheorem{theorem}{\indent\sc Theorem}[section] % 見出しはスモールキャップ
\newtheorem{lemma}[theorem]{\indent\sc Lemma}
 
 \newtheorem{proposition}[theorem]{\indent\sc Proposition}
 
\theoremstyle{definition} % ローマン体に変更

 \newtheorem{remark}[theorem]{\indent\sc Remark}

 \makeatletter
\def\address#1#2{\begingroup
 \noindent\parbox[t]{7.8cm}{%
 \small{\scshape\ignorespaces#1}\par\vskip1ex
 \noindent\small{\itshape E-mail address}%
 \/: #2\par\vskip4ex}\hfill%
 \endgroup}%
 \makeatother
\title{%\uppercase{
Global existence for a system of quasi-linear wave equations 
in $3$D satisfying the weak null condition}
 \author{
 \textsc{Kunio Hidano and Kazuyoshi Yokoyama} %著者名
}
 \date{} %日付けは記入しない
\begin{document}

 \maketitle

%%%%%%%%%%%%%%% 脚注 %%%%%%%%%%%%%%%%
 \footnote{ %2010 MSC numbers
 }
 \footnote{ %key words and phrases
 }
 \footnote{ %Thanks
The authors are grateful to the referee for 
reading the manuscript carefully, 
pointing out some careless mistakes, 
and bringing \cite{DLY}, \cite{LT} to their attention. 
It is a pleasure to thank Professor Soichiro Katayama. 
His suggestion was helpful in simplifying the proof of 
Proposition A.1. 
Special thanks also go to Professor Dongbing Zha for 
bringing \cite{KMS} to the attention of the authors. 
The authors are partly supported by 
the Grant-in-Aid for Scientific Research (C) 
15K04955, Japan Society for the Promotion of Science. 
}
%%%%%%%%%%%%%%%%%%%%%%%%%%%%%%%%%%%%%%%%%
\begin{abstract}
We show global existence of small solutions 
to the Cauchy problem for a system of 
quasi-linear wave equations in three space dimensions. 
The feature of the system 
lies in that it satisfies the weak null condition, 
though we permit the presence of some quadratic nonlinear terms 
which do not satisfy the null condition.  
Due to the presence of such quadratic terms, 
the standard argument no longer works 
for the proof of global existence. 
To get over this difficulty, 
we extend the ghost weight method of Alinhac 
so that it works for the system under consideration. 
The original theorem of Alinhac 
for the scalar unknowns is also refined. 
\end{abstract}

% \section*{Introduction} %Introductionに見出し番号を付ける場合は*をとる
%Acknowledgments may be included at the end of the Introduction.

\section{Introduction}
In \cite{Al2010}, Alinhac studied the Cauchy problem 
for the quasi-linear wave equation in three space dimensions 
of the form
\begin{equation}
\begin{cases}
\displaystyle{
\partial_t^2 u-\Delta u
+
G^{\alpha\beta\gamma}
(\partial_\gamma u)
(\partial_{\alpha\beta}^2 u)=0,
\,\,t>0,\,x\in{\mathbb R}^3
             },\\
%%%%%%%%%%%%%%%%%%%%%%
\displaystyle{
u(0)=\varphi_0,\,\,\partial_t u(0)=\varphi_1
             },
\end{cases}
\end{equation}
where $u:(0,\infty)\times{\mathbb R}^3\to{\mathbb R}$, 
$G^{\alpha\beta\gamma}\in{\mathbb R}$, 
$G^{\alpha\beta\gamma}=G^{\beta\alpha\gamma}$, 
$\partial_0:=\partial_t$, 
and $\varphi_0,\,\varphi_1\in C_0^\infty({\mathbb R}^3)$. 
Here, and in the following, 
when the same index is above and below, 
summation is assumed 
from $0$ to $3$ for 
$\alpha$, $\beta$, $\gamma$. 
Using the remarkable energy inequality 
(see page 92 of \cite{Al2010}), 
he proved the following:

\begin{theorem}[Alinhac]
%$({\rm Alinhac})$ 
Suppose that the coefficients 
$G^{\alpha\beta\gamma}$ satisfy the null condition$:$ 
\begin{equation}
G^{\alpha\beta\gamma}X_\alpha X_\beta X_\gamma=0\,\,
\mbox{for any $X=(X_0,\dots,X_3)$ with 
$X_0^2=X_1^2+X_2^2+X_3^2.$}
\end{equation}
Set 
\begin{equation}
R_*:=
\inf
\bigl\{\,r>0\,:
\,{\rm supp}\,\{\varphi_0,\varphi_1\}
\subset 
\{
x\in{\mathbb R}^3:|x|<r
\}
\bigr\}.
\end{equation}
Then, there exist constants $C>0$, 
$0<\varepsilon<1$ depending on 
$\{G^{\alpha\beta\gamma}\}$ and $R_*$ 
with 
$\varepsilon\to 0$ as $R_*\to\infty$ such that 
if 
\begin{equation}
W_4(u(0))\leq\varepsilon,
\end{equation}
then 
the Cauchy problem $(1.1)$ admits 
a unique global, smooth solution $u(t,x)$ 
satisfying 
\begin{equation}
W_4(u(t))\leq
CW_4(u(0)).
\end{equation}
\end{theorem}

\begin{remark}
In page 89 of \cite{Al2010}, 
only the nonlinear terms involving 
the spatial derivatives were considered. 
This was just for simplicity.
\end{remark}

Here and in the following discussion, we use the notation: 
\begin{eqnarray}
& &
E_1(u(t))
:=
\frac12
\int_{{\mathbb R}^3}
\bigl(
(\partial_t u(t,x))^2
+
|\nabla u(t,x)|^2
\bigr)dx,\\
& &
W_\kappa(u(t))
:=
\sum_{|a|\leq \kappa-1}
E_1^{1/2}(Z^a u(t)),
\quad \kappa=2,3,\dots
\end{eqnarray}
By $Z$, we mean any of the operators 
$\partial_\alpha$ $(\alpha=0,\dots,3)$, 
$\Omega_{ij}:=x_i\partial_j-x_j\partial_i$ 
$(1\leq i<j\leq 3)$, 
$L_k:=x_k\partial_t+t\partial_k$ 
$(k=1,2,3)$, 
and $S:=t\partial_t+x\cdot\nabla$. 
Also, for a multi-index $a$, 
$Z^a$ stands for any product 
of the $|a|$ these operators. 
We remark that $\partial_t^ku(0,x)$ for $k=2,3,4$ 
can be calculated with the help of the equation (1.1), 
and thus the quantity $W_4(u(0))$ appearing in (1.4) 
is determined by the given initial data 
$(\varphi_0,\varphi_1)$. 

The novelty of this theorem due to Alinhac lies in that 
as for the size of data, 
we have only to assume that 
$W_4(u(0))$ is small enough. 
It should be compared with the fact that 
if we employ the standard energy inequality 
for variable-coefficient hyperbolic operators 
(see, e.g., (6.3.6) of \cite{Hor}) 
and the Klainerman-Sobolev inequalities (see Lemma 2.4 below) 
together with the good commutation relations (2.1)--(2.2) below, 
we can obtain:
%%%%%%%%%%%%%%%
\begin{proposition}
Suppose the null condition $(1.2)$. 
Then, there exist constants $C>0$, $0<\varepsilon<1$ 
depending only on the coefficients $G^{\alpha\beta\gamma}$ 
such that if 
\begin{equation}
W_4(u(0))\exp\bigl(CW_5(u(0))\bigr)\leq\varepsilon,
\end{equation}
then the Cauchy problem $(1.1)$ admits a unique 
global smooth solution $u(t,x)$ satisfying
\begin{equation}
W_4(u(t))
\leq
W_4(u(0))\exp\bigl(CW_5(u(0))\bigr),\,\,
W_5(u(t))
\leq
CW_5(u(0))(1+t)^{C\varepsilon}.
\end{equation}
\end{proposition}
The proof of Proposition 1.3 uses the two semi-norms, 
allowing the higher-order one 
$W_5(u(t))$ to grow slowly over time, 
and bounding the lower-order one $W_4(u(t))$ 
uniformly in time with the help of 
the estimation lemma (see Lemma 2.3 below). 
This is why we need the size condition (1.8), which is obviously 
more restrictive than (1.4). 
On the other hand, the constant $\varepsilon$ of Proposition 1.3 
is independent of $R_*$ (see (1.3)). 
In this regard, Proposition 1.3 has an advantage over Theorem 1.1. 

The purpose of this paper is twofold. 
Firstly, we aim at refining Theorem 1.1 by 
showing that 
it is in fact possible to choose the constant $\varepsilon$ 
independently of $R_*$. 
Secondly, we intend to generalize Theorem 1.1 
to the result for a class of diagonal systems 
of quasi-linear wave equations, such as 
\begin{equation}
\begin{cases}
\displaystyle{
\partial_t^2 u_1
-
\Delta u_1
+
G_1^{11,\alpha\beta\gamma}
(\partial_\gamma u_1)
(\partial_{\alpha\beta}^2 u_1)
+
G_1^{21,\alpha\beta\gamma}
(\partial_\gamma u_2)
(\partial_{\alpha\beta}^2 u_1)
             }\\
%%%%%%%%%%%%%%%%%%%%%
\hspace{0.2cm}
\displaystyle{
+
H_1^{11,\alpha\beta}
(\partial_\alpha u_1)
(\partial_\beta u_1)
+
H_1^{12,\alpha\beta}
(\partial_\alpha u_1)
(\partial_\beta u_2)
+
H_1^{22,\alpha\beta}
(\partial_\alpha u_2)
(\partial_\beta u_2)
=0
               },\\
%%%%%%%%%%%%%%%%%
\displaystyle{
\partial_t^2 u_2
-
\Delta u_2
+
G_2^{12,\alpha\beta\gamma}
(\partial_\gamma u_1)
(\partial_{\alpha\beta}^2 u_2)
+
G_2^{22,\alpha\beta\gamma}
(\partial_\gamma u_2)
(\partial_{\alpha\beta}^2 u_2)
                }\\
%%%%%%%%%%%%%%%%%%%%%%%
\hspace{0.2cm}
\displaystyle{
+
H_2^{12,\alpha\beta}
(\partial_\alpha u_1)
(\partial_\beta u_2)
+
H_2^{11,\alpha\beta}
(\partial_\alpha u_1)
(\partial_\beta u_1)
+
H_2^{22,\alpha\beta}
(\partial_\alpha u_2)
(\partial_\beta u_2)
=0.
               }
\end{cases}
\end{equation}
%%%%%%%%%%%%%%%%%%%%%%%%%%%%
Supposing the null condition 
for all the coefficients of the first equation,  
%$(G_1^{11,\alpha\beta\gamma})$, 
%$(G_1^{21,\alpha\beta\gamma})$, 
%$(H_1^{11,\alpha\beta})$, $(H_1^{12,\alpha\beta})$, 
%and $(H_1^{22,\alpha\beta})$, 
and supposing the null condition only on 
$\{G_2^{22,\alpha\beta\gamma}\}$ and 
$\{H_2^{22,\alpha\beta}\}$ 
as for the coefficients of the second equation, 
we prove:

\begin{theorem}
Suppose the symmetry condition: there hold 
$G_1^{11,\alpha\beta\gamma}=G_1^{11,\beta\alpha\gamma}$, 
$G_1^{21,\alpha\beta\gamma}=G_1^{21,\beta\alpha\gamma}$, 
and 
$G_2^{12,\alpha\beta\gamma}=G_2^{12,\beta\alpha\gamma}$, 
$G_2^{22,\alpha\beta\gamma}=G_2^{22,\beta\alpha\gamma}$ 
for all $\alpha,\beta,\gamma=0,\dots,3$. 
Also, suppose 
\begin{align}
&
G_1^{11,\alpha\beta\gamma}
X_\alpha X_\beta X_\gamma
=
G_1^{21,\alpha\beta\gamma}
X_\alpha X_\beta X_\gamma
=
G_2^{22,\alpha\beta\gamma}
X_\alpha X_\beta X_\gamma
=0,\\
%%%%%%%%%%%%%%%%%%%%
&
H_1^{11,\alpha\beta}X_\alpha X_\beta
=
H_1^{12,\alpha\beta}X_\alpha X_\beta
=
H_1^{22,\alpha\beta}X_\alpha X_\beta
=
H_2^{22,\alpha\beta}X_\alpha X_\beta
=0
\end{align}
for any $X=(X_0,\dots,X_3)\in{\mathbb R}^4$ 
with 
$X_0^2=X_1^2+X_2^2+X_3^2$. 
Let 
$0<\eta<1/6$, 
$0<\delta<1/6$ so that 
$\eta+2\delta<1/2$. 
%2017/4/15 Assumptions on \eta, \delta updated
Then, there exist constants 
$C>0$, $0<\varepsilon<1$ 
depending only on 
the coefficients of the system $(1.10)$, 
$\delta$, and $\eta$ 
such that 
if compactly supported smooth data satisfy 
$W_4(u_1(0))+W_4(u_2(0))<\varepsilon$, 
then the Cauchy problem for $(1.10)$ 
admits a unique global smooth solution 
$(u_1(t,x),u_2(t,x))$ satisfying for all 
$t>0$, $T>0$
\begin{align}
&
W_4(u_1(t))
+
(1+t)^{-\delta}
W_4(u_2(t))\\
&
\hspace{0.4cm}
+
\sum_{i=1}^3
\sum_{|a|\leq 3}
\biggl(
\|
\langle t-r\rangle^{-(1/2)-\eta}
T_i Z^a u_1
\|_{L^2((0,\infty)\times{\mathbb R}^3)}\nonumber\\
&
\hspace{2.5cm}
+
(1+T)^{-\delta}
\|
\langle t-r\rangle^{-(1/2)-\eta}
T_i Z^a u_2
\|_{L^2((0,T)\times{\mathbb R}^3)}
\biggr)\nonumber\\
&
\hspace{0.2cm}
\leq
C\bigl(
W_4(u_1(0))+W_4(u_2(0))
\bigr).\nonumber
\end{align}
Here $T_i=\partial_i+(x_i/|x|)\partial_t$. 
\end{theorem}
Here, and in the following as well, 
we use the standard notation 
$\langle p\rangle=\sqrt{1+p^2}$. 
Note that 
by choosing the trivial data 
$u_2(0,x)=\partial_t u_2(0,x)=0$ 
and assuming $H^{11,\alpha\beta}_2=0$ for 
all $\alpha,\beta$, and thus considering the trivial solution 
$u_2(t,x)\equiv 0$, 
we can go back to the wave equation for the scalar unknowns 
with the nonlinear terms more general than those of (1.1)
\begin{equation}
\partial_t^2 u-\Delta u
+
G^{\alpha\beta\gamma}
(\partial_\gamma u)
(\partial_{\alpha\beta}^2 u)
+
H^{\alpha\beta}
(\partial_\alpha u)
(\partial_\beta u)
=0,
\,\,t>0,\,x\in{\mathbb R}^3
\end{equation}
and thus obtain: 
\begin{theorem}
Suppose the symmetry condition 
$G^{\alpha\beta\gamma}=G^{\beta\alpha\gamma}$. 
Also, suppose the null condition: there holds 
\begin{equation}
G^{\alpha\beta\gamma}X_\alpha X_\beta X_\gamma
=
H^{\alpha\beta}X_\alpha X_\beta=0
\end{equation}
for any $X=(X_0,\dots,X_3)\in{\mathbb R}^4$ 
with $X_0^2=X_1^2+X_2^2+X_3^2$.  
Let $0<\eta<1/6$ be fixed. 
Then, there exist constants 
$C>0$, $0<\varepsilon<1$ depending only on $\eta$, 
$\{G^{\alpha\beta\gamma}\}$ and 
$\{H^{\alpha\beta}\}$ 
such that 
if compactly supported smooth data satisfy 
\begin{equation}
W_4(u(0))\leq\varepsilon,
\end{equation}
then 
the Cauchy problem $(1.14)$ admits 
a unique global, smooth solution $u(t,x)$ 
satisfying 
\begin{equation}
\sup_{t>0}W_4(u(t))
+
\sum_{i=1}^3
\sum_{|a|\leq 3}
\|
\langle t-r\rangle^{-(1/2)-\eta}
T_i Z^a u
\|_{L^2((0,\infty)\times{\mathbb R}^3)}
\leq
CW_4(u(0)).
\end{equation}
Moreover, the estimate
\begin{equation}
\sup_{t>0}W_5(u(t))
\leq
CW_5(u(0))
%\exp\{CW_4(u(0))\}
\end{equation}
also holds.
\end{theorem}
Theorem 1.5 significantly improves Theorem 1.1 
because the constant $\varepsilon$ no longer depends upon $R_*$. 
Also, we permit the presence of the ``semi-linear term'' 
$H^{\alpha\beta}(\partial_\alpha u)(\partial_\beta u)$. 
Theorem 1.5 improves Theorem 1.1 of Wang \cite{FangWang} too, 
where in the absence of such semi-linear terms, 
global existence of solutions to (1.1) and 
uniform (in time) boundedness of $W_5(u(t))$ 
was shown under the assumption that $W_5(u(0))\leq\varepsilon$. 
(Recall that in Theorem 1.5, we have only assumed $W_4(u(0))$ is small.) 
We remark that the constant $\varepsilon$ in Theorem 1.1 of \cite{FangWang} 
also depends upon $R_*$. 
In addition, we should mention the recent paper of Zha \cite{Zha}, which 
the authors became aware of while preparing the present manuscript. 
In \cite{Zha}, again in the absence of the semi-linear terms, 
global existence of solutions to (1.1) and 
uniform (in time) boundedness of $W_7(u(t))$ 
was shown under the assumption that $W_7(u(0))\leq\varepsilon$. 
(Strictly speaking, the definition of the ``generalized energy norm'' in \cite{Zha} 
is slightly different from that of ours.) 
We should note that the constant $\varepsilon$ 
in the theorem of Zha no longer depends upon $R_*$. 
Finally, we remark that the bound (1.18) tells us that 
the ``grow up'' (or ``blow up at $t=\infty$'') suggested by 
the second estimate in (1.9) in fact never occurs. 

The null condition was originally introduced 
by Christodoulou \cite{Ch} and Klainerman \cite{Kl86} 
independently, 
as the sufficient condition on the form of quadratic nonlinear terms 
under which 
the Cauchy problem for diagonal systems of quasi-linear wave equations 
with quadratic and fairly general higher-order nonlinear terms 
admits global solutions whenever initial data are smooth and small. 
%%%%%%%%%On the weak null condition%%%%%%%%%%%%%
It is worthwhile to remark that 
the system (1.10) does not satisfy the null condition 
but does satisfy the weak null condition 
which was introduced by Lindblad and Rodnianski \cite{LR}. 
Now let us recall the condition. 
We anticipate the asymptotic expansion 
\begin{align}
 u(t,x) \sim \frac{\varepsilon U(q, s, \omega)}{|x|}\\
 %\label{05141450}\\
 \intertext{as $|x| \to \infty$ and $|x| \sim t$, where}
  q = |x|-t,\quad s = \varepsilon \log |x|, \quad \omega = \frac{x}{|x|}.\nonumber
\end{align}
Then we obtain a system of equations for $U=(U_i)$, 
the {\itshape asymptotic system}, by substituting (1.19) 
into the original system and extracting the main terms. 
We say that a system of nonlinear wave equations satisfies 
the {\itshape weak null condition} 
if the following conditions are satisfied 
(see \cite{LR} for details):
\begin{itemize}
\item The corresponding asymptotic system admits a global solution for all initial data at $s=0$ decaying sufficiently fast in $q$.
\item The global solution grows in $s$ at most exponentially together with its derivatives. 
\end{itemize}
In the case of (1.10), we see that the asymptotic system is
\begin{align}
 &2\partial_s \partial_q U_1 =0,\\
% \label{05131857}\\
 & 2\partial_s \partial_q U_2 
 = A^{12}_{2, 12}(\omega)(\partial_q U_1)(\partial_q^2 U_2)
 + A^{12}_{2, 11}(\omega)(\partial_q U_1)(\partial_q U_2)
 + A^{11}_{2,11}(\omega)(\partial_q U_1)^2,\\
%\label{05131858}\\
\intertext{where}
 &A^{12}_{2, 12}(\omega) 
 = G^{12, \alpha \beta \gamma}_2 
 \widehat{\omega}_\alpha\widehat{\omega}_\beta\widehat{\omega}_\gamma,\,\,
A^{12}_{2, 11}(\omega)
=
H^{12, \alpha \beta}_{2} \widehat{\omega}_\alpha\widehat{\omega}_\beta,\,\,
A^{11}_{2,11}(\omega)
=
H^{11,\alpha \beta}_{2}
\widehat{\omega}_\alpha\widehat{\omega}_\beta,
%\label{05141630}
\end{align}
and $\widehat{\omega}=(-1, \omega)$. 
Note that the quadratic terms such as 
$G^{11,\alpha \beta \gamma}_1(\partial_\gamma u_1)(\partial^2_{\alpha \beta}u_1)$,
$\dots$, 
$H^{22, \alpha \beta}_2(\partial_\alpha u_2)(\partial_\beta u_2)$ 
whose coefficients satisfy the null condition (1.11), (1.12) 
are negligible compared to the main terms. 
Note also the presence of the terms 
$A^{12}_{2,12}(\omega)(\partial_q U_1)(\partial^2_q U_2)$, 
$A^{12}_{2,11}(\omega)(\partial_q U_1)(\partial_q U_2)$ in 
(1.21), %\ref{05131858}), 
which shows a feature of our study. 
The asymptotic system of the type 
\begin{align*}
 2\partial_s \partial_q U_1 =0,\quad 2\partial_s \partial_q U_2 = (\partial_q U_1)^2
\end{align*}
has already appeared in the study of 
the Einstein vacuum equations in harmonic coordinates 
(see \cite{LR}). 
As far as the present authors know, 
global existence for (1.10) with small initial data 
whose corresponding asymptotic system is (1.20)--(1.21) 
has not been proved until now. 
Actually, it is not difficult to find a solution of (1.20)--(1.21) 
%(\ref{05131857})-(\ref{05131858}) 
obeying 
\begin{align}
 U_1\vert_{s=0} = F_1,\quad U_2\vert_{s=0} = F_2
\end{align}
and growing in $s$ at most exponentially. 
Indeed, we get from (1.20) 
%(\ref{05131857}) 
a solution of the form $U_1=F_1(q, \omega)$, 
which is independent of $s$. 
Substituting this into (1.21), 
%(\ref{05131858}), 
we have
\begin{align}
&
 2\partial_s \partial_q U_2
  - 
  A^{12}_{2, 12}(\omega)\partial_q F_1(q, \omega) \partial_q^2 U_2\\
&  
\hspace{0.2cm}
 = 
 A^{12}_{2, 11}(\omega)(\partial_q F_1(q, \omega))(\partial_q U_2)
 +
 A^{11}_{2,11}(\omega)\left(\partial_q F_1(q, \omega)\right)^2.\nonumber
\end{align}
This can be regarded as 
a linear first order partial differential equation for $\partial_q U_2$. 
By the standard argument (see, e.g., page 13 of \cite{Hor}), 
we can find a global solution growing in $s$ at most exponentially, 
which means that the system (1.10) satisfies the weak null condition. 
Since the weak null condition 
strongly raises the possibility of 
global existence for 
the original system, 
it is very important to investigate whether 
there exist global solutions to (1.10) for small, smooth data. 
To the best of the present authors' knowledge, 
there exist a few works verifying the prediction that 
the system violating the null condition but satisfying the weak null condition 
actually admits global solutions for small, smooth data. 
See \cite{Lind1992}, \cite{Al2003}, 
\cite{Al2006}, \cite{LR2005}, \cite{Lind2008}, \cite{DLY}, \cite{KMS}, 
and \cite{LT}. 
In this paper, we are going to prove global existence for (1.10). 
%%%%%%%%%%%%%%%%%%%%%%%%%%%%%%%%%%%%%%%%%%%%

The standard argument of showing global existence results such as 
Proposition 1.3 requires that a lower-order energy remain 
small as $t\to\infty$ (see (1.9)). 
Since $W_4(u_2(t))$ may possibly grow as $t\to\infty$ 
(see (1.13)), 
the standard argument does not apply to 
the proof of global existence for the system (1.10). 
We instead employ the ghost weight technique of showing 
Theorem 1.1 above (see Chapter 9 of \cite{Al2010}). 
Naturally, some modifications are necessary 
so as to choose $\varepsilon$ 
independently of $R_*$ 
and also to study the system (1.10). 
The constant $\varepsilon$ of Theorem 1.1 depends upon $R_*$ 
because on pages 94 and 95 Alinhac uses the inequality
\begin{equation}
(1+t)(1+|t-r|)^{-1/2}|\phi(t,x)|
\leq
C\sum_{|a|\leq 2}
\|\partial_x Z^a\phi(t,\cdot)\|_{L^2({\mathbb R}^3)}
\end{equation}
for smooth functions 
$\phi(t,x)$ with 
${\rm supp}\,\phi(t,\cdot)\subset\{x\in{\mathbb R}^3:|x|<t+R\}$ 
for some $R>0$, 
where the constant $C$ on the right-hand side of (1.25) 
does depend upon on $R$. 
We avoid using (1.25) 
and instead employ the trace-type inequality (see Lemma 2.5 below), 
which plays a key role in 
estimating the nonlinear terms on the region 
$\{x\in{\mathbb R}^3:|x|>(t/2)+1\}$ 
differently from how Alinhac did with the use of (1.25). 
For completeness, in Appendix we prove the inequality (1.25) 
in general space dimensions. 

In order to show global existence of solutions to the system (1.10), 
we need the two Alinhac-type energy estimates; 
one is for the hyperbolic operator 
\begin{equation}
\partial_t^2-\Delta
+
g_1^{\alpha\beta\gamma}(\partial_\gamma w(t,x))
\partial_{\alpha\beta}^2
\end{equation} 
with the coefficients $\{g_1^{\alpha\beta\gamma}\}$ satisfying the null condition 
(see (3.7) below, where we will actually 
consider a little more general operator), 
and the other is for the operator 
\begin{equation}
\partial_t^2-\Delta
+
g_2^{\alpha\beta\gamma}(\partial_\gamma v(t,x))
\partial_{\alpha\beta}^2
+
g_3^{\alpha\beta\gamma}(\partial_\gamma w(t,x))
\partial_{\alpha\beta}^2
\end{equation}
with 
the coefficients $\{g_3^{\alpha\beta\gamma}\}$ satisfying the null condition 
and with the coefficients $\{g_2^{\alpha\beta\gamma}\}$ failing in 
satisfying the null condition 
(see (3.8) below). 
Though in page 92 of \cite{Al2010} 
the remarkable energy inequality 
is stated for such hyperbolic operators as (1.26) 
under the assumption that 
the variable coefficient $w(t,x)$ satisfies 
\begin{equation}
\sum_{|a|\leq 3}
\|\partial Z^a w(t,\cdot)\|_{L^2({\mathbb R}^3)}
\leq
C_0\varepsilon,
\end{equation}
it is worthwhile noting that 
the method of Alinhac is in fact considerably robust and 
it equally works also for $w(t,x)$ behaving 
\begin{equation}
\sum_{|a|\leq 3}
\|\partial Z^a w(t,\cdot)\|_{L^2({\mathbb R}^3)}
\leq
C_0\varepsilon
(1+t)^\delta
\end{equation}
for some $0\leq\delta<1/2$. 
See Proposition 3.1 below. 
This key fact is very helpful in 
studying the system (1.10) 
whose solutions, in particular $u_2$, 
will be far from behaving like free solutions 
as $t\to\infty$. 
We also remark that 
even for (1.27), 
we can obtain a useful energy estimate 
by suitably modifying the proof of the lemma on page 92 of \cite{Al2010}. 
(See Proposition 3.2 below.) 
In this way, we generalize the method of Alinhac 
so as to prove the global existence for the system (1.10) 
satisfying the weak null condition. 

This paper is organized as follows. 
In the next section, key facts on the nonlinear terms 
satisfying the null condition are stated and 
useful inequalities such as 
the Sobolev type or the trace type are collected. 
In Section 3, in order to prove Theorem 1.4 
the energy estimate is carried out and 
the desired a priori estimate is obtained. 
In Section 4, we show (1.18) to complete the proof of 
Theorem 1.5. 
In Appendix, we prove (1.25).
%%%%%%%%%%%%%%%%%%%%%%%%%%%%%%%%%%%%%%%%%%%%%%%
%%%%%%%%%%%%%%%%%%%%%%%%%%%%%%%%%%%%%%%%%%%
%%%%%%%%%%%%%%%%%%%%%%%%%%%%%%%%%%%%%%%%%%%
\section{Preliminaries}
As explained in Section 1, 
in addition to the usual partial differential operators 
$\partial_0:=\partial/\partial t$ and 
$\partial_i:=\partial/\partial x_i$ 
$(i=1,2,3)$, 
we use $\Omega_{ij}:=x_i\partial_j-x_j\partial_i$, 
$L_k:=x_k\partial_0+t\partial_k$ 
$(1\leq i<j\leq 3,\,k=1,2,3)$ 
and $S=t\partial_0+x\cdot\nabla$. 
The set of these $11$ 
differential operators is denoted by 
$Z=\{Z_0,\dots,Z_{10}\}
=\{\partial_0,\partial_1,\dots,S\}$. 
For a multi-index $a=(a_0,\dots,a_{10})$, 
we set $Z^a:=Z_0^{a_0}\cdots Z_{10}^{a_{10}}$. 

We first remark several results concerning 
commutation relations. 
Let $[\cdot,\cdot]$ be the commutator: 
$[A,B]:=AB-BA$. 
It is easy to verify that 
\begin{eqnarray}
& &
[Z_i,\partial_t^2-\Delta]=0\,\,\,\mbox{for $i=0,\dots,9$},\,\,\,
[S,\partial_t^2-\Delta]=-2(\partial_t^2-\Delta),\\
& &
[Z_j,\partial_k]
=
\sum_{i=0}^3 C^{j,k}_i\partial_i,\,\,\,j=0,\dots,10,\,\,k=0,\dots,3.
\end{eqnarray}
Here $C^{j,k}_i$ denotes a constant depending on 
$i$, $j$, and $k$. 
These can be verified easily. 

The next lemma states that the null form is preserved 
under the differentiation. 
\begin{lemma}
Suppose that $\{G^{\alpha\beta\gamma}\}$ and 
$\{H^{\alpha\beta}\}$ satisfy the null condition 
$($see $(1.11)$, $(1.12)$ above$)$. 
For any $Z_i$ $(i=0,\dots,10)$, 
the equality 
\begin{align}
&
Z_i
G^{\alpha\beta\gamma}
(\partial_\gamma v)
(\partial_{\alpha\beta}^2 w)\\
&=
G^{\alpha\beta\gamma}
(\partial_\gamma Z_i v)
(\partial_{\alpha\beta}^2 w)
+
G^{\alpha\beta\gamma}
(\partial_\gamma v)
(\partial_{\alpha\beta}^2 Z_i w)
+
{\tilde G}_i^{\alpha\beta\gamma}
(\partial_\gamma v)
(\partial_{\alpha\beta}^2 w)\nonumber
\end{align}
holds 
with the new coefficients 
$\{{\tilde G}_i^{\alpha\beta\gamma}\}$ 
also satisfying the null condition. 
Also, the equality 
\begin{equation}
Z_i
H^{\alpha\beta}
(\partial_\alpha v)
(\partial_\beta w)
=
H^{\alpha\beta}
(\partial_\alpha Z_i v)
(\partial_\beta w)
+
H^{\alpha\beta}
(\partial_\alpha v)
(\partial_\beta Z_i w)
+
{\tilde H}_i^{\alpha\beta}
(\partial_\alpha v)
(\partial_\beta w)
\end{equation}
holds 
with the new coefficients 
$\{{\tilde H}_i^{\alpha\beta}\}$ 
also satisfying the null condition. 
\end{lemma}
%%%%%%%%%%%%%%%%%%%%%%

For the proof, see, e.g., page 91 of \cite{Al2010}.

%%%%%%%%%%%%%%%%%%%%%%%%%%%%%%%%%
It is possible to show the following lemma 
essentially in the same way as in pages 90--91 of \cite{Al2010}. 
(See also Lemma 2.3 of \cite{LNS}. 
The present authors are inspired by 
these arguments in \cite{Al2010} and \cite{LNS}.) 
Using (2.5) and (2.6), we will later exploit the fact that 
for local solutions $u$, 
the special derivatives $T_i u$ have 
better space-time $L^2$-integrability 
and improved time decay property of their $L^\infty({\mathbb R}^n)$-norms. 
%%%%%%%%%%%%%%%%%%%%%%%%%%%%%%%%%%%%%%
\begin{lemma}Suppose that 
$\{G^{\alpha\beta\gamma}\}$, $\{H^{\alpha\beta}\}$ satisfy 
the null condition. Then, we have
\begin{align}
&
|
G^{\alpha\beta\gamma}
(\partial_\gamma v)
(\partial_{\alpha\beta}^2 w)
|
\leq
C
(
|T v|
|\partial^2 w|
+
|\partial v|
|T\partial w|
),\\
&
|
H^{\alpha\beta}
(\partial_\alpha v)
(\partial_\beta w)
|
\leq
C
(
|T v|
|\partial w|
+
|\partial v|
|T w|
).
\end{align}
\end{lemma}
%%%%%%%%%%%%%%%%%%%%%%%%%%%%%%%%

Here, and in the following, we use the notation
%%%%%%%%%%%%%%%%%%%%%%%%%%%%%%%%%%%
\begin{equation}
|Tv|
:=
\biggl(
\sum_{k=1}^3
|T_k v|^2
\biggr)^{1/2},\quad
|T\partial v|
:=
\biggl(
\sum_{k=1}^3
\sum_{\gamma=0}^3
|T_k\partial_\gamma v|^2
\biggr)^{1/2}
\end{equation}
%%%%%%%%%%%%%%%%%%%%%%%%%%%%%%%%%%
%%%%%%%%%%%%%%%%%%%%%%%%%%%%%%%%%%%%%%
\begin{proof}
We may focus on (2.5), 
because the proof of (2.6) is similar. 
Using the representation 
$\partial_i=T_i-\omega_i\partial_t$ 
$(i=1,2,3,\,\omega_i=x_i/|x|)$ 
and 
setting $T_0=0$, $\omega_0=-1$, 
we have 
%%%%%%%%Use Yokoyama-san's draft%%%%%%%%%%
\begin{align}
&
G^{\alpha\beta\gamma}
(\partial_\gamma v)
(\partial_{\alpha\beta}^2 w)\\
&
\hspace{0.2cm}
=
G^{\alpha\beta\gamma}
(T_\gamma v)
(\partial_{\alpha\beta}^2 w)
-
G^{\alpha\beta\gamma}
\omega_\gamma
(\partial_t v)
(\partial_{\alpha\beta}^2 w)\nonumber\\
&
\hspace{0.2cm}
=
G^{\alpha\beta\gamma}
(T_\gamma v)
(\partial_{\alpha\beta}^2 w)
-
G^{\alpha\beta\gamma}
\omega_\gamma
(\partial_t v)
(T_\alpha\partial_\beta w)
+
G^{\alpha\beta\gamma}
\omega_\alpha\omega_\gamma
(\partial_t v)
(\partial_t\partial_\beta w)\nonumber\\
&
\hspace{0.2cm}
=
G^{\alpha\beta\gamma}
(T_\gamma v)
(\partial_{\alpha\beta}^2 w)
-
G^{\alpha\beta\gamma}
\omega_\gamma
(\partial_t v)
(T_\alpha\partial_\beta w)\nonumber\\
&
\hspace{0.56cm}
+
G^{\alpha\beta\gamma}
\omega_\alpha\omega_\gamma
(\partial_t v)
(T_\beta\partial_t w)
-
G^{\alpha\beta\gamma}
\omega_\alpha\omega_\beta\omega_\gamma
(\partial_t v)
(\partial_t^2 w).\nonumber
\end{align}
%%%%%%%%%%%%%%%%%%%%%%%%%%%%%%%%%%%%%%%%
We find that owing to the null condition, 
the last term on the right-hand side above vanishes, 
which shows (2.5).
\end{proof}
%%%%%%%%%%%%%%%%%%%%%%%%%%%%%%%%%%

The next lemma says that 
the null condition creates cancellation which 
allows us to handle the quadratic nonlinear terms 
as higher-order ones in terms of time decay. 
%%%%%%%%%%%%%%%%%%%%%%%%%%%%%%%%%%%%%%%
\begin{lemma}If $\{G^{\alpha\beta\gamma}\}$ satisfies the 
null condition, 
then 
\begin{equation}
|
G^{\alpha\beta\gamma}
(\partial_\gamma v)
(\partial_{\alpha\beta}^2 w)
|
\leq
C(1+t)^{-1}
\sum_{|a|=1}
\bigl(
|Z^a v|
|\partial^2 w|
+
|\partial v|
|\partial Z^a w|
\bigr)
\end{equation}
holds. Also, if $\{H^{\alpha\beta}\}$ satisfies the null condition, 
then 
\begin{equation}
|
H^{\alpha\beta}
(\partial_\alpha v)
(\partial_\beta w)
|
\leq
C(1+t)^{-1}
\sum_{|a|=1}
\bigl(
|Z^a v|
|\partial w|
+
|\partial v|
|Z^a w|
\bigr)
\end{equation}
holds.  
\end{lemma}
%%%%%%%%%%%%%%%%%%%%%%%%%%%%%%%%

Here, and in the following, 
we use the standard notation
\begin{equation}
|\partial v|
:=
\biggl(
\sum_{\gamma=0}^3
|\partial_\gamma v|^2
\biggr)^{1/2},
\quad
|\partial^2 v|
:=
\biggl(
\sum_{\alpha,\beta=0}^3
|\partial_{\alpha\beta}^2 v|^2
\biggr)^{1/2}.
\end{equation}
%%%%%%%%%%%%%%%%%%%%%%%%%%%%%%%%%%%
\begin{proof}
Using the fact
\begin{equation}
|Tv(t,x)|
\leq
C(1+t)^{-1}
\sum_{|a|=1}
|Z^a v(t,x)|
\end{equation}
(see page 91 of \cite{Al2010}, and see also (3.26) below) 
together with the commutation relation (2.2), 
we can derive (2.9)--(2.10) from (2.5)--(2.6).
\end{proof}
%%%%%%%%%%%%%%%%%%%%%%%%%%%%%%%%%%%%%%%%%%%%%%%%%%%%%%%%%%
%%%%%%%%%%%%%%%%%%%%%%%%%%%%%%%%%%%%%%%%%%%%%%%5

The following lemma is concerned with Sobolev-type 
inequalities. 
\begin{lemma}
For any smooth function 
$v(t,x)\in C^\infty((0,\infty)\times{\mathbb R}^3)$ 
decaying sufficiently fast as $|x|\to\infty$, 
the inequality 
\begin{equation}
(1+t+|x|)(1+|t-|x||)^{1/2}|v(t,x)|
\leq
C\sum_{|a|\leq 2}\|Z^a v(t,\cdot)\|_{L^2({\mathbb R}^3)}
\end{equation}
holds. 
Moreover, for any $p$ with $2\leq p\leq 6$ 
there exists a positive constant $C$ depending on $p$ 
such that the inequality
\begin{equation}
\|(1+t+|\cdot|)^{2((1/2)-(1/p))}(1+|t-|\cdot||)^{(1/2)-(1/p)}v(t,\cdot)\|
_{L^p({\mathbb R}^3)}
\leq
C\sum_{|a|\leq 1}\|Z^a v(t,\cdot)\|_{L^2({\mathbb R}^3)}
\end{equation}
holds. 
\end{lemma}
\begin{proof}
See \cite{Kl87} for (2.13), and \cite{GV} for (2.14). 
\end{proof}
We also use the following trace-type inequality.
\begin{lemma}
Let $n\geq 2$. 
For any $s$ with $1/2<s<n/2$, 
there exists a positive constant 
depending only on $n,\,s$ such that 
if $v=v(x)$ decays sufficiently fast as $|x|\to\infty$, 
then the inequality
\begin{equation}
r^{(n/2)-s}
\|v(r\cdot)\|_{L^p(S^{n-1})}
\leq
C\|v\|_{{\dot H}^{s}({\mathbb R}^n)},\,\,
\frac{n-1}{p}=\frac{n}2-s
\end{equation}
holds, where $r:=|x|$.
\end{lemma}
\begin{proof}
The proof uses the trace-type inequality 
due to Hoshiro \cite{Hoshiro} (for $n\geq 3$) 
and Fang and Wang \cite{FangChengbo} (for $n=2$), 
together with the Sobolev embedding on $S^{n-1}$. 
See, e.g., Proposition 2.4 of \cite{HWY}.
\end{proof}
%%%%%%%%%%%%%%%%%%%%%%%%%%%%%%
\section{Proof of Theorem 1.4}
Since the second order quasi-linear hyperbolic system (1.10) 
can be written in the form of 
the first order quasi-linear symmetric hyperbolic system 
(see, e.g., (5.9) of Racke \cite{Racke}), 
the standard local existence theorem 
(see, e.g., Theorem 5.8 of \cite{Racke}) 
applies to the Cauchy problem 
for (1.10). 
In what follows, 
we assume that 
the initial data are smooth, compactly supported, 
and small so that 
\begin{equation}
W_4(u_1(0))
+
W_4(u_2(0))
\leq
\min
\left\{
\frac{\varepsilon_1}{C_0},
\frac{\varepsilon_2}{C_0},
\frac{C_0}{6C_3}
\right\}
\end{equation}
may hold. 
(See Propositions 3.1, 3.2, and 3.3 below for the constants 
$\varepsilon_1$, $\varepsilon_2$, $C_0$, and $C_3$.) 
We thus know that 
a unique, smooth solution to (1.10) exists 
at least for a short time interval, 
and it is compactly supported at any fixed time 
by the finite speed of propagation. 
We therefore also know that the set 
\begin{align}
\{
T>0:
&
\,\mbox{There exists a smooth solution 
$(u_1(t,x),u_2(t,x))$ to $(1.10)$ defined for}\\
& 
\,
\mbox{all 
$(t,x)\in (0,T)\times{\mathbb R}^3$ 
satisfying}\nonumber\\
&
\sup_{0<t<T}
\bigl(
W_4(u_1(t))
+
(1+t)^{-\delta}W_4(u_2(t))
\bigr)<\infty
\}\nonumber
\end{align} 
is not empty.  
We define $T^*$  as the supremum of this set. 
(If we assume $T^*<\infty$, 
then we will get contradiction 
and hence obtain global existence.) 
Using this local smooth solution 
$u=(u_1,u_2)$, 
we next define 
\begin{equation}
A(u(t))
:=
W_4(u_1(t))
+
(1+t)^{-\delta}
W_4(u_2(t)).
\end{equation}
Thanks to the compactness of the support 
of the initial data 
and the finiteness of the propagation speed of the solution, 
we can easily verify the important property that 
\begin{equation}
A(u(t))\in 
C([0,T^*)).
\end{equation}
Let $C_0\geq 2$ be the constant determined later 
(see Proposition 3.3 below). 
Owing to (3.4) and the absolute continuity of integral, 
the set 
\begin{align}
\{
T\in (0,T^*)\,:
&
\sup_{0<t<T}A(u(t))
+
\sum_{|a|\leq 3}
\biggl(
\int_0^T\!\!\!\int_{{\mathbb R}^3}
\langle t-r\rangle^{-1-2\eta}
\sum_{i=1}^3
|T_i Z^a u_1(t,x)|^2dtdx
\biggr)^{1/2}\\
&
\hspace{0.56cm}
+
\sum_{|a|\leq 3}
\sup_{0<t<T}
(1+t)^{-\delta}
\biggl(
\int_0^t\!\!\!\int_{{\mathbb R}^3}
\langle \tau-r\rangle^{-1-2\eta}
\sum_{i=1}^3
|T_i Z^a u_2(\tau,x)|^2d\tau dx
\biggr)^{1/2}\nonumber\\
&
\hspace{0.2cm}
\leq
C_0\bigl(
W_4(u_1(0))+W_4(u_2(0))
\bigr)
\}
\nonumber
\end{align}
is not empty. Let us define $T_*$ as the supremum of this set. 
By definition, we know $T_*\leq T^*$. 
We are going to show 
the key a priori estimate: 
we actually have for all $T\in (0,T_*)$
\begin{align}
&\sup_{0<t<T}A(u(t))
+
\sum_{|a|\leq 3}
\biggl(
\int_0^T\!\!\!\int_{{\mathbb R}^3}
\langle t-r\rangle^{-1-2\eta}
\sum_{i=1}^3
|T_i Z^a u_1(t,x)|^2dtdx
\biggr)^{1/2}\\
&
\hspace{1cm}
+
\sum_{|a|\leq 3}
\sup_{0<t<T}
(1+t)^{-\delta}
\biggl(
\int_0^t\!\!\!\int_{{\mathbb R}^3}
\langle \tau-r\rangle^{-1-2\eta}
\sum_{i=1}^3
|T_i Z^a u_2(\tau,x)|^2d\tau dx
\biggr)^{1/2}
%\sum_{|a|\leq 3}
%(1+T)^{-\delta}
%\biggl(
%\int_0^T\!\!\!\int_{{\mathbb R}^3}
%\langle t-r\rangle^{-1-2\eta}
%\sum_{i=1}^3
%|T_i Z^a u_2(t,x)|^2dtdx
%\biggr)^{1/2}
\nonumber\\
&
\hspace{0.2cm}
\leq
\frac{2}{3}C_0\bigl(
W_4(u_1(0))+W_4(u_2(0))
\bigr).\nonumber
\end{align}
To obtain the estimate (3.6), 
we use the ghost weight technique of Alinhac. 
Define the linearized hyperbolic operators 
$P_1$ and $P_2$ as follows:
\begin{align}
&
P_1:=
\partial_t^2-\Delta
+
g_1^{11,\alpha\beta\gamma}
(\partial_\gamma \varphi)
\partial_{\alpha\beta}^2
+
g_1^{21,\alpha\beta\gamma}
(\partial_\gamma \psi)
\partial_{\alpha\beta}^2,\\
&
P_2:=
\partial_t^2-\Delta
+
g_2^{12,\alpha\beta\gamma}
(\partial_\gamma \varphi)
\partial_{\alpha\beta}^2
+
g_2^{22,\alpha\beta\gamma}
(\partial_\gamma \psi)
\partial_{\alpha\beta}^2
\end{align}
for 
$\varphi,\psi\in C^{\infty}((0,T)\times{\mathbb R}^3)$. 

\begin{proposition}
Suppose the symmetry condition$:$ 
there hold 
$g_1^{11,\alpha\beta\gamma}=g_1^{11,\beta\alpha\gamma}$, 
$g_1^{21,\alpha\beta\gamma}=g_1^{21,\beta\alpha\gamma}$ 
for all $\alpha,\beta,\gamma=0,\dots,3$. 
Also, suppose the null condition$:$ 
there holds 
\begin{equation}
g_1^{11,\alpha\beta\gamma}X_\alpha X_\beta X_\gamma
=
g_1^{21,\alpha\beta\gamma}X_\alpha X_\beta X_\gamma
=0
\end{equation}
for all $X=(X_0,\dots,X_3)$ 
satisfying 
$X_0^2=X_1^2+X_2^2+X_3^2$. 

Let $\eta>0$, $\delta>0$ satisfy $\eta+\delta<1/2$. 
Then, there exist two constants 
$\varepsilon_1$ and $C_1$ 
$(0<\varepsilon_1<1, C_1>0)$ depending 
only on $\eta$, $\delta$ and the coefficients $g_1^{11,\alpha\beta\gamma}$, 
$g_1^{21,\alpha\beta\gamma}$ such that 
if the functions $\varphi$ and $\psi$ satisfy 
\begin{equation}
\sup_{0<t<T}
\bigl(
W_4(\varphi(t))
+
(1+t)^{-\delta}W_4(\psi(t))
\bigr)
\leq
\varepsilon_1,
\end{equation}
then the estimate
\begin{align}
&
\sup_{0<t<T}
\|\partial u(t,\cdot)\|_{L^2({\mathbb R}^3)}
+
\biggl(
\int_0^T\!\!\!\int_{{\mathbb R}^3}
\langle t-r\rangle^{-1-2\eta}
\sum_{i=1}^3
|T_i u(t,x)|^2dtdx
\biggr)^{1/2}\\
&
\hspace{0.2cm}
\leq
C_1
\biggl(
\|\partial u(0,\cdot)\|_{L^2({\mathbb R}^3)}
+
\int_0^T
\|P_1 u(t,\cdot)\|_{L^2({\mathbb R}^3)}dt
\biggr)\nonumber
\end{align}
holds.
\end{proposition}
%%%%%%%%%%%%%%%%%%%%%%%%%%%%%%%%%%5
\begin{proposition}
Suppose the symmetry condition$:$ 
there hold 
$g_2^{12,\alpha\beta\gamma}=g_2^{12,\beta\alpha\gamma}$, 
$g_2^{22,\alpha\beta\gamma}=g_2^{22,\beta\alpha\gamma}$ 
for all $\alpha,\beta,\gamma=0,\dots,3$. 
Also, suppose the null condition only on $\{g_2^{22,\alpha\beta\gamma}\}$. 
Let $\eta>0$, $\delta>0$ satisfy $\eta+\delta<1/2$. 
Then, there exist two constants 
$\varepsilon_2$ and $C_2$ 
$(0<\varepsilon_2<1, C_2>0)$ depending 
only on $\eta$, $\delta$ and the coefficients $g_2^{12,\alpha\beta\gamma}$, 
$g_2^{22,\alpha\beta\gamma}$ such that 
if the functions $\varphi$ and $\psi$ satisfy 
\begin{equation}
\sup_{0<t<T}
\bigl(
W_4(\varphi(t))
+
(1+t)^{-\delta}W_4(\psi(t))
\bigr)
\leq
\varepsilon_2, 
\end{equation}
then the estimate
\begin{align}
&
\sup_{0<t<T}
(1+t)^{-\delta}
\|\partial u(t,\cdot)\|_{L^2({\mathbb R}^3)}\\
&
\hspace{0.56cm}
+
\sup_{0<t<T}
(1+t)^{-\delta}
\left(
\int_0^t\!\!\!\int_{{\mathbb R}^3}
\langle \tau-r\rangle^{-1-2\eta}
\sum_{i=1}^3
|T_i u(\tau,x)|^2d\tau dx
\right)^{1/2}\nonumber\\
&
\hspace{0.2cm}
\leq
C_2
\|\partial u(0,\cdot)\|_{L^2({\mathbb R}^3)}
+
C_2
\sup_{0<t<T}
\left(
(1+t)^{-2\delta}
\int_0^t
(1+\tau)^{\delta}
\|P_2 u(\tau,\cdot)\|_{L^2({\mathbb R}^3)}d\tau
\right)
\nonumber
\end{align}
holds.
\end{proposition}
\begin{proof}
We first prove Proposition 3.1. 
We basically follow the argument on pages 
92 to 94 of \cite{Al2010}. 
Let $a=a(t-r)$ be chosen later. 
In the same way as on page 93 of \cite{Al2010}, 
we get by using 
$T_i=\partial_i+\omega_i\partial_t$ 
$(\omega_i:=x_i/|x|,\,i=1,2,3)$
\begin{align}
e^a
(\square u)\partial_t u
&=
\frac12
\partial_t
\bigl\{
e^a
\bigl(
(\partial_t u)^2
+
|\nabla u|^2
\bigr)
\bigr\}\\
&
-\nabla
\cdot
\{\cdots\}
-\frac12
e^a
a'(t-r)
\sum_{i=1}^3
(T_i u)^2,\nonumber
\end{align}
%%%%%%%%%%%%%%%%%%%%%%%%%%
\begin{align}
&
e^a
g_1^{11,\alpha\beta\gamma}
(\partial_\gamma\varphi)
(\partial_{\alpha\beta}^2 u)
(\partial_t u)\\
&
\hspace{0.2cm}
=
-\frac12
\partial_t
\{
e^a
g_1^{11,\alpha\beta\gamma}
(\partial_\gamma\varphi)
(\partial_{\beta} u)
(\partial_\alpha u)
-2e^a
g_1^{11,0\beta\gamma}
(\partial_\gamma\varphi)
(\partial_{\beta} u)
(\partial_t u)
\}\nonumber\\
&
\hspace{0.7cm}
+\nabla\cdot\{\cdots\}
+
\frac12
e^a
\{
g_1^{11,\alpha\beta\gamma}
(\partial_{t\gamma}^2\varphi)
(\partial_{\beta} u)
(\partial_\alpha u)
-2
g_1^{11,\alpha\beta\gamma}
(\partial_{\alpha\gamma}^2\varphi)
(\partial_{\beta} u)
(\partial_t u)
\}\nonumber\\
&
\hspace{0.7cm}
+\frac12
e^a
a'(t-r)
\{
g_1^{11,\alpha\beta\gamma}
(\partial_{\gamma} \varphi)
(\partial_{\beta} u)
(\partial_\alpha u)
-2
g_1^{11,0\beta\gamma}
(\partial_{\gamma} \varphi)
(\partial_{\beta} u)
(\partial_t u)\nonumber\\
&
\hspace{3.7cm}
+
2\omega_i
g_1^{11,i\beta\gamma}
(\partial_{\gamma} \varphi)
(\partial_{\beta} u)
(\partial_t u)
\}.\nonumber
\end{align}
Here, and in the following as well, 
because it vanishes after integration over ${\mathbb R}^3$, 
we write such a ``harmless'' term as $\nabla\cdot\{\cdots\}$. 
Obviously, we also get the equality 
similar to (3.15), with 
$g_1^{11,\alpha\beta\gamma}$ and $\varphi$ 
replaced by 
$g_1^{21,\alpha\beta\gamma}$ and $\psi$, respectively. 
Therefore, we have
\begin{align}
&
e^a(P_1 u)(\partial_t u)\\
&
\hspace{0.2cm}
=
\frac12
\partial_t
\bigl\{
e^a
\bigl(
(\partial_t u)^2
+
|\nabla u|^2\nonumber\\
&
\hspace{2.0cm}
-
g_1^{11,\alpha\beta\gamma}
(\partial_{\gamma} \varphi)
(\partial_{\beta} u)
(\partial_\alpha u)
+2
g_1^{11,0\beta\gamma}
(\partial_{\gamma} \varphi)
(\partial_{\beta} u)
(\partial_t u)\nonumber\\
&
\hspace{2.0cm}
-
g_1^{21,\alpha\beta\gamma}
(\partial_{\gamma} \psi)
(\partial_{\beta} u)
(\partial_\alpha u)
+2
g_1^{21,0\beta\gamma}
(\partial_{\gamma} \psi)
(\partial_{\beta} u)
(\partial_t u)
\bigr)
\bigr\}\nonumber\\
&
\hspace{0.58cm}
+\nabla\cdot\{\cdots\}
+
e^aq,\nonumber
\end{align}
where 
$q=q_1-a'(t-r)(q_2/2)$,
\begin{align}
q_1
&
=
\frac12
g_1^{11,\alpha\beta\gamma}
(\partial_{t\gamma}^2 \varphi)
(\partial_{\beta} u)
(\partial_\alpha u)
-
g_1^{11,\alpha\beta\gamma}
(\partial_{\alpha\gamma}^2 \varphi)
(\partial_{\beta} u)
(\partial_t u)\\
&
\hspace{0.56cm}
+
\frac12
g_1^{21,\alpha\beta\gamma}
(\partial_{t\gamma}^2 \psi)
(\partial_{\beta} u)
(\partial_\alpha u)
-
g_1^{21,\alpha\beta\gamma}
(\partial_{\alpha\gamma}^2 \psi)
(\partial_{\beta} u)
(\partial_t u),\nonumber
\end{align}
%%%%%%%%%%%%%%%%%
and, for $\omega_0=-1$
\begin{align}
q_2
&=
\sum_{i=1}^3(T_i u)^2\\
&
\hspace{0.40cm}
-
g_1^{11,\alpha\beta\gamma}
(\partial_{\gamma} \varphi)
(\partial_{\beta} u)
(\partial_\alpha u)
+
2g_1^{11,\alpha\beta\gamma}
(\partial_{\gamma} \varphi)
(\partial_{\beta} u)
(-\omega_\alpha)
(\partial_t u)\nonumber\\
&
\hspace{0.40cm}
-
g_1^{21,\alpha\beta\gamma}
(\partial_{\gamma} \psi)
(\partial_{\beta} u)
(\partial_\alpha u)
+
2g_1^{21,\alpha\beta\gamma}
(\partial_{\gamma} \psi)
(\partial_{\beta} u)
(-\omega_\alpha)
(\partial_t u).\nonumber
\end{align}
We will follow page 93 of \cite{Al2010} to deal with $q_1$, 
while to treat $q_2$, we need to proceed differently in part 
from page 94 of \cite{Al2010}. 
Writing $g^{\alpha\beta\gamma}=g_1^{11,\alpha\beta\gamma}$ 
for short, using $\partial_i=T_i-\omega_i\partial_t$ 
$(\omega_i=x_i/|x|,\,i=1,2,3)$, 
and setting $T_0=0$, $\omega_0=-1$ as before, 
we have for the first term on the right-hand side of (3.17)
\begin{align}
&
g^{\alpha\beta\gamma}
(\partial_{t\gamma}^2 \varphi)
(\partial_\beta u)
(\partial_\alpha u)\\
&
\hspace{0.2cm}
=
g^{\alpha\beta\gamma}
(T_\gamma\partial_t \varphi)
(\partial_\beta u)
(\partial_\alpha u)
-
g^{\alpha\beta\gamma}
\omega_\gamma
(\partial_t^2 \varphi)
(\partial_\beta u)
(\partial_\alpha u)\nonumber\\
&
\hspace{0.2cm}
=
g^{\alpha\beta\gamma}
(T_\gamma\partial_t \varphi)
(\partial_\beta u)
(\partial_\alpha u)
-
g^{\alpha\beta\gamma}
\omega_\gamma
(\partial_t^2 \varphi)
(T_\beta u)
(\partial_\alpha u)\nonumber\\
&
\hspace{0.56cm}
+
g^{\alpha\beta\gamma}
\omega_\beta
\omega_\gamma
(\partial_t^2 \varphi)
(\partial_t u)
(\partial_\alpha u)\nonumber\\
&
\hspace{0.2cm}
=
g^{\alpha\beta\gamma}
(T_\gamma\partial_t \varphi)
(\partial_\beta u)
(\partial_\alpha u)
-
g^{\alpha\beta\gamma}
\omega_\gamma
(\partial_t^2 \varphi)
(T_\beta u)
(\partial_\alpha u)\nonumber\\
&
\hspace{0.56cm}
+
g^{\alpha\beta\gamma}
\omega_\beta
\omega_\gamma
(\partial_t^2 \varphi)
(\partial_t u)
(T_\alpha u)
-
g^{\alpha\beta\gamma}
\omega_\alpha
\omega_\beta
\omega_\gamma
(\partial_t^2 \varphi)
(\partial_t u)^2.\nonumber
\end{align}
%%%%%%%%%%%%%%%%%%%%%%%%%%%%%%%%%%%%%%%%
Since $g_1^{11,\alpha\beta\gamma}
\omega_\alpha
\omega_\beta
\omega_\gamma=0$ 
due to the null condition, 
we finally obtain
\begin{equation}
|g_1^{11,\alpha\beta\gamma}
(\partial_{t\gamma}^2 \varphi)
(\partial_{\beta} u)
(\partial_\alpha u)|
\leq
C
|T\partial_t\varphi|
|\partial u|^2
+
C
|\partial_t^2\varphi|
|Tu|
|\partial u|.
\end{equation}
(For the definition of 
$|Tv|$ and $|\partial v|$, 
see (2.7), (2.11).) 
Similarly, we can get for the second term on the right-hand side of (3.17)
\begin{equation}
|g_1^{11,\alpha\beta\gamma}
(\partial_{\alpha\gamma}^2 \varphi)
(\partial_{\beta} u)
(\partial_t u)|
\leq
C
|T\partial\varphi|
|\partial u|
|\partial_t u|
+
C
|\partial_t^2 \varphi|
|Tu|
|\partial_t u|.
\end{equation}
(For the definition of 
$|T\partial v|$ and $|\partial^2 v|$, 
see (2.7), (2.11).) 
Obviously, for the third and fourth terms 
on the right-hand side of (3.17), 
we can obtain the inequality 
similar to (3.20) and (3.21), respectively, 
naturally with $\varphi$ replaced by $\psi$. 

For the second and third terms on the right-hand side of 
(3.18), we can proceed as above to obtain
\begin{align}
&|
g_1^{11,\alpha\beta\gamma}
(\partial_\gamma\varphi)
(\partial_\beta u)
(\partial_\alpha u)
|
\leq
C
|T\varphi|
|\partial u|^2
+
C
|\partial_t\varphi|
|Tu|
|\partial u|,\\
&
|
g_1^{11,\alpha\beta\gamma}
(\partial_\gamma\varphi)
(\partial_\beta u)
(-\omega_\alpha)
(\partial_t u)
|
\leq
C
|\partial_t\varphi|
|Tu|
|\partial_t u|
+
C
|T\varphi|
|\partial u|
|\partial_t u|.
\end{align}
%%%%%%%%%%%%%%%%%%%%%%%%%%%%%%%%%%%
Naturally, the inequalities similar to (3.22) and (3.23) hold 
for the fourth and fifth terms on the right-hand side of (3.18), 
respectively, with $\varphi$ replaced by $\psi$. 

We get by (3.19)--(3.23)
\begin{equation}
q_1
\geq
-C
(
|T\partial\varphi|
+
|T\partial\psi|
)
|\partial u|^2
-C
(
|\partial_t^2\varphi|
+
|\partial_t^2\psi|
)
|Tu|
|\partial u|
\end{equation}
and
\begin{equation}
q_2
\geq
\sum_{i=1}^3(T_i u)^2
-C
(
|T\varphi|
+
|T\psi|
)
|\partial u|^2
-C
(
|\partial_t\varphi|
+
|\partial_t\psi|
)
|Tu|
|\partial u|
\end{equation}
for a positive constant $C$. 
To continue the estimate of $q_1$ and $q_2$, 
we use the remarkable improvement 
of point-wise decay of the special derivatives $T_k v$. 
Using
\begin{equation}
T_k
=
\frac{1}{t}
\bigl(
L_k-\omega_k(r-t)\partial_t
\bigr)
=
\frac{1}{r}
\bigl(
L_k+(r-t)\partial_k
\bigr)
\end{equation}
and
\begin{equation}
\partial_t
=
\frac{\displaystyle{tS-\sum_{j=1}^3x_jL_j}}{t^2-r^2},\quad
\partial_k
=
\frac{\displaystyle{tL_k+\sum_{j=1}^3x_j\Omega_{kj}-x_kS}}{t^2-r^2},
\end{equation}
we get
\begin{align}
|T_k\partial_\gamma\varphi|
&\leq
C(1+t+r)^{-1}
\sum_{|a|=1}|Z^a \partial_\gamma\varphi|\\
&
\leq
C(1+t+r)^{-2}
(1+|t-r|)^{-1/2}W_4(\varphi(t)),\nonumber\\
|T_k\partial_\gamma\psi|
&\leq
C(1+t+r)^{-2+\delta}
(1+|t-r|)^{-1/2}
\bigl(
(1+t)^{-\delta}
W_4(\psi(t))
\bigr),
\end{align}
which yield
\begin{equation}
-C
(
|T\partial\varphi|
+
|T\partial\psi|
)
|\partial u|^2
\geq
-C(1+t)^{-2+\delta}
\Phi(t)
|\partial u|^2,
\end{equation}
where, and in the following as well, we use the notation
\begin{equation}
\Phi(t):=
W_4(\varphi(t))
+
(1+t)^{-\delta}
W_4(\psi(t)).
\end{equation}
Furthermore, 
since we have by (2.13)
\begin{align}
|\partial_t^2\varphi|
+
|\partial_t^2\psi|
&
\leq
C(1+t+r)^{-1+\delta}
(1+|t-r|)^{-1/2}
\Phi(t)\\
&
\leq
C(1+t+r)^{-1+\delta+\eta}
(1+|t-r|)^{-(1/2)-\eta}
\Phi(t),\nonumber
\end{align}
we obtain 
\begin{align}
&
-C
(
|\partial_t^2\varphi|
+
|\partial_t^2\psi|
)
|Tu|
|\partial u|
\\
&
\geq
-C(1+|t-r|)^{-(1/2)-\eta}
|Tu|
\Phi(t)^{1/2}
(1+t+r)^{-1+\delta+\eta}
|\partial u|
\Phi(t)^{1/2}\nonumber\\
&
\geq
-C
(1+|t-r|)^{-1-2\eta}
|Tu|^2
\Phi(t)
-C
(1+t+r)^{-2+2\delta+2\eta}
|\partial u|^2
\Phi(t).\nonumber
\end{align}
Combining (3.24) with (3.30), (3.33), 
we get
\begin{align}
q_1\geq&
-C(1+|t-r|)^{-1-2\eta}\Phi(t)|Tu|^2\\
&
-C(1+t)^{-2+\delta}\Phi(t)|\partial u|^2
-C(1+t)^{-2+2\delta+2\eta}\Phi(t)|\partial u|^2.\nonumber
\end{align}
To estimate $q_2$ from below, 
we proceed differently from 
what Alinhac did 
by using (1.25) 
on page 94 of \cite{Al2010}. 
We first divide $(0,\infty)\times{\mathbb R}^3$ 
into the two pieces: 
$D_{int}:=\{(t,x):|x|<(t/2)+1\}$ 
and its complement $D_{ext}$. 
Using (2.13) and (3.18), 
we easily get for $(t,x)\in D_{int}$
\begin{equation}
q_2
\geq
|Tu|^2
-C(1+t)^{-(3/2)+\delta}
\Phi(t)|\partial u|^2.
\end{equation}
On the other hand, 
for $(t,x)\in D_{ext}$, we use (3.25). 
Since we have by the Sobolev embedding on $S^2$ and (2.15) 
\begin{align}
|T_k\varphi|
&
\leq
C(1+t+r)^{-1}\sum_{|a|=1}|Z^a\varphi|\\
&
\leq
C(1+t+r)^{-1}\sum_{|a|=1}
\biggl(
\|Z^a\varphi(r\cdot)\|_{L^4(S^2)}
+
\sum_{1\leq i<j\leq 3}
\|\Omega_{ij}Z^a\varphi(r\cdot)\|_{L^4(S^2)}
\biggr)\nonumber\\
&
\leq
C(1+t+r)^{-1}r^{-1/2}W_4(\varphi(t)),\nonumber
\end{align}
we get for $(t,x)\in D_{ext}$
\begin{equation}
|T\varphi|
+
|T\psi|
\leq
C(1+t)^{-(3/2)+\delta}\Phi(t),
\end{equation}
which yields for $(t,x)\in D_{ext}$
\begin{align}
q_2
\geq
&|Tu|^2
-C(1+t)^{-(3/2)+\delta}\Phi(t)|\partial u|^2\\
&
-C(1+|t-r|)^{-1-2\eta}\Phi(t)|Tu|^2
-C(1+t)^{-2+2\delta+2\eta}\Phi(t)|\partial u|^2.\nonumber
\end{align}
Now we are ready to finish the proof of Proposition 3.1. 
Let us choose the function $a=a(\rho)$ $(\rho\in{\mathbb R})$ as
$$
a(\rho)
=
-\int_0^\rho
(1+|s|)^{-1-2\eta}ds
$$
so that $a'(\rho)=-(1+|\rho|)^{-1-2\eta}$ may hold. 
Remark that there exists a positive constant $c_1$ 
such that $c_1\leq e^a\leq c_1^{-1}$, 
because $a$ is bounded. 
We then observe for $a=a(t-r)$
\begin{align}
\int_{{\mathbb R}^3}
e^aqdx
&=
\int_{{\mathbb R}^3}
e^a
(q_1+(1+|t-r|)^{-1-2\eta}q_2)dx\\
&
\geq
\bigl(
c_1-c_2\Phi(t)
\bigr)
\int_{{\mathbb R}^3}
(1+|t-r|)^{-1-2\eta}|Tu|^2dx\nonumber\\
&
\hspace{0.4cm}
-c_3
\max\{(1+t)^{-(3/2)+\delta},(1+t)^{-2+2\delta+2\eta}\}
\Phi(t)
\int_{{\mathbb R}^3}
|\partial u|^2dx, \nonumber
\end{align}
where $c_2$ and $c_3$ are positive constants. 
Suppose that 
$\sup_{0<t<T}\Phi(t)$ is sufficiently small so that
\begin{align*}
&c_1-c_2\sup_{0<t<T}\Phi(t)
\geq
\frac{c_1}{2},\\
&
-c_3\int_0^\infty\max\{(1+t)^{-(3/2)+\delta},(1+t)^{-2+2\delta+2\eta}\}dt
\sup_{0<t<T}\Phi(t)\geq -c_4
\end{align*}
($c_4$ is a sufficiently small positive constant). 
By integrating (3.16) over 
$(0,T)\times{\mathbb R}^3$ 
and then using the Young inequality to get 
\begin{align*}
&\int_0^T\|P_1u(t)\|_{L^2({\mathbb R}^3)}dt
\sup_{0<t<T}\|\partial u(t)\|_{L^2({\mathbb R}^3)}\\
&
\hspace{0.2cm}
\leq
C\left(\int_0^T\|P_1u(t)\|_{L^2({\mathbb R}^3)}dt\right)^2
+
C^{-1}\left(\sup_{0<t<T}\|\partial u(t)\|_{L^2({\mathbb R}^3)}\right)^2
\end{align*}
for a sufficiently large positive constant $C$, 
and finally moving the second term on the right-hand side above to the other side 
of the resulting inequality, 
we obtain (3.11). 

Let us next prove Proposition 3.2. 
We assume the null condition on 
$\{g_2^{22,\alpha\beta\gamma}\}$ 
but not on $\{g_2^{12,\alpha\beta\gamma}\}$. 
We can proceed as in (3.19)--(3.25), 
except the difference that we must 
take account of the terms
\begin{equation}
g_2^{12,\alpha\beta\gamma}
\omega_\alpha
\omega_\beta
\omega_\gamma
(\partial_t^i\varphi)
(\partial_t u)^2,\quad
i=1,2
\end{equation}
which do not always vanish. 
We therefore get 
\begin{equation}
q_1
\geq
-C|\partial_t^2\varphi|(\partial_t u)^2
-C
(
|T\partial\varphi|
+
|T\partial\psi|
)
|\partial u|^2
-C
(
|\partial_t^2\varphi|
+
|\partial_t^2\psi|
)
|Tu|
|\partial u|
\end{equation}
and
\begin{equation}
q_2
\geq
\sum_{i=1}^3(T_i u)^2
-C|\partial_t\varphi|(\partial_t u)^2
-C
(
|T\varphi|
+
|T\psi|
)
|\partial u|^2
-C
(
|\partial_t \varphi|
+
|\partial_t \psi|
)
|Tu|
|\partial u|
\end{equation}
in place of (3.24)--(3.25). 
These yield
\begin{align}
\int_{{\mathbb R}^3}
e^a q dx
&\geq
\bigl(
c_1-c_5\Phi(t)
\bigr)
\int_{{\mathbb R}^3}
(1+|t-r|)^{-1-2\eta}
|Tu|^2dx\\
&
\hspace{0.4cm}
-C
(1+t)^{-1+2\delta}
\Phi(t)
\left(
(1+t)^{-2\delta}
\int_{{\mathbb R}^3}
|\partial u(t,x)|^2dx
\right),\nonumber
\end{align}
where $c_5$ is a positive constant. 
Suppose that $\sup_{0<t<T}\Phi(t)$ is small so that 
$$
c_1-c_5\sup_{0<t<T}\Phi(t)\geq \frac{c_1}{2}.
$$ 
Then, we get for $0<t<T$ (cf. (3.16))
\begin{align}
&
\left[
\frac12
\int_{{\mathbb R}^3}
e^a
\biggl(
(\partial_t u)^2
+
|\nabla u|^2
\right.\\
&
\hspace{2cm}
-
g_2^{12,\alpha\beta\gamma}
(\partial_{\gamma} \varphi)
(\partial_{\beta} u)
(\partial_\alpha u)
-2
g_2^{12,0\beta\gamma}
(\partial_{\gamma} \varphi)
(\partial_{\beta} u)
(\partial_t u)\nonumber\\
&
\hspace{2cm}
\left.
-
g_2^{22,\alpha\beta\gamma}
(\partial_{\gamma} \psi)
(\partial_{\beta} u)
(\partial_\alpha u)
-2
g_2^{22,0\beta\gamma}
(\partial_{\gamma} \psi)
(\partial_{\beta} u)
(\partial_t u)
\biggr)dx
\right]_{\tau=0}^{\tau=t}\nonumber\\
&
\hspace{0.6cm}
+\frac12
c_1\int_0^t
\int_{{\mathbb R}^3}
(1+|\tau-r|)^{-1-2\eta}
|Tu(\tau,x)|^2dxd\tau\nonumber\\
&
\hspace{0.6cm}
-c_6(1+t)^{2\delta}
\left(
\sup_{0<\tau<T}
\Phi(\tau)
\right)
\sup_{0<\tau<T}
\left(
(1+\tau)^{-2\delta}
\int_{{\mathbb R}^3}
|\partial u(\tau,x)|^2dx
\right)\nonumber\\
&
\hspace{0.2cm}
\leq
\int_0^t
\|P_2u(\tau)\|_{L^2({\mathbb R}^3)}
\|\partial_t u(\tau)\|_{L^2({\mathbb R}^3)}d\tau.\nonumber
\end{align}
Suppose further that $c_6\sup_{0<t<T}\Phi(t)$ is small enough. 
We divide both sides of (3.44) by $(1+t)^{2\delta}$ 
and use the Young inequality to get
\begin{align*}
&
(1+t)^{-2\delta}
\int_0^t
\|P_2u(\tau)\|_{L^2({\mathbb R}^3)}
\|\partial_t u(\tau)\|_{L^2({\mathbb R}^3)}d\tau\\
&
\hspace{0.2cm}
\leq
(1+t)^{-2\delta}
\int_0^t
(1+\tau)^{\delta}
\|P_2u(\tau)\|_{L^2({\mathbb R}^3)}
d\tau
\sup_{0<\tau<t}
(1+\tau)^{-\delta}\|\partial_t u(\tau)\|_{L^2({\mathbb R}^3)}\\
&
\hspace{0.2cm}
\leq
C
\sup_{0<t<T}
\left(
(1+t)^{-2\delta}
\int_0^t
(1+\tau)^{\delta}
\|P_2u(\tau)\|_{L^2({\mathbb R}^3)}
d\tau
\right)^2\\
&
\hspace{0.56cm}
+
C^{-1}
\left(
\sup_{0<\tau<T}
(1+\tau)^{-\delta}
\|\partial_t u(\tau)\|_{L^2({\mathbb R}^3)}
\right)^2, 
\end{align*}
where $C$ is a sufficiently large constant. 
Since $C^{-1}$ is small enough, 
we finally obtain (3.13). 
\end{proof}
%%%%%%%%%%%%%%%%%%%%%%%%%
Next, let us carry out the energy estimate for 
$u_1$ and $u_2$. 
\begin{proposition}
Set $C_0:=2\max\{1,\,C_1,\,C_2\}$ 
for the constants $C_1$ and $C_2$ 
$($see $(3.11)$, $(3.13))$. 
Suppose that 
the initial data satisfies 
\begin{equation}
C_0
\bigl(
W_4(u_1(0))+W_4(u_2(0))
\bigr)
\leq
\min\{\varepsilon_1,\,\varepsilon_2\}
\end{equation}
$($see $(3.10)$, $(3.12)$ for $\varepsilon_1$, $\varepsilon_2)$ 
and the local solution $(u_1,u_2)$ to $(1.10)$ satisfies for some $T>0$
\begin{align}
&\sup_{0<t<T}A(u(t))
+
\sum_{|a|\leq 3}
\biggl(
\int_0^T\!\!\!\int_{{\mathbb R}^3}
\langle t-r\rangle^{-1-2\eta}
\sum_{i=1}^3
|T_i Z^a u_1(t,x)|^2dtdx
\biggr)^{1/2}\\
&
\hspace{1cm}
+
\sum_{|a|\leq 3}
\sup_{0<t<T}
(1+t)^{-\delta}
\biggl(
\int_0^t\!\!\!\int_{{\mathbb R}^3}
\langle \tau-r\rangle^{-1-2\eta}
\sum_{i=1}^3
|T_i Z^a u_2(\tau,x)|^2d\tau dx
\biggr)^{1/2}\nonumber\\
&
\hspace{0.2cm}
\leq
C_0\bigl(
W_4(u_1(0))+W_4(u_2(0))
\bigr).\nonumber
\end{align}
Then, there exists a constant $C_3>0$ such that 
\begin{align}
&
\sup_{0<t<T}A(u(t))
+
\sum_{|a|\leq 3}
\biggl(
\int_0^T\!\!\!\int_{{\mathbb R}^3}
\langle t-r\rangle^{-1-2\eta}
\sum_{i=1}^3
|T_i Z^a u_1(t,x)|^2dtdx
\biggr)^{1/2}\\
&
\hspace{1cm}
+
\sum_{|a|\leq 3}
\sup_{0<t<T}
(1+t)^{-\delta}
\biggl(
\int_0^t\!\!\!\int_{{\mathbb R}^3}
\langle \tau-r\rangle^{-1-2\eta}
\sum_{i=1}^3
|T_i Z^a u_2(\tau,x)|^2d\tau dx
\biggr)^{1/2}
\nonumber\\
&
\hspace{0.2cm}
\leq
C_1W_4(u_1(0))
+
C_2W_4(u_2(0))
+
C_3\bigl(
W_4(u_1(0))+W_4(u_2(0))
\bigr)^2.\nonumber
\end{align}
\end{proposition}
\begin{proof}
Let us start with the energy estimate for $u_1$. 
In view of Proposition 3.1 with 
$u$ replaced by $Z^a u_1$ $(|a|\leq 3)$, 
we need to deal with the integral of
\begin{equation}
\|
\bigl(
\partial_t^2-\Delta
+
\sum_{i=1}^2
G_1^{i1,\alpha\beta\gamma}
(\partial_\gamma u_i)\partial^2_{\alpha\beta}
\bigr)
Z^a u_1(t)
\|_{L^2({\mathbb R}^3)}
\end{equation}
$(|a|\leq 3)$ from $0$ to $T$. 
Moreover, thanks to (2.1), (2.3)--(2.4), 
it suffices to bound the integral from $0$ to $T$ of 
\begin{align}
&
\sum_{{|b|+|c|\leq 3}\atop{|c|\leq 2}}
\left\|
G_{bc}^{\alpha\beta\gamma}
(\partial_\gamma Z^b u_1)
(\partial^2_{\alpha\beta}Z^c u_1)
\right\|_{L^2({\mathbb R}^3)},\\
&
\sum_{{|b|+|c|\leq 3}\atop{|c|\leq 2}}
\left\|
{\tilde G}_{bc}^{\alpha\beta\gamma}
(\partial_\gamma Z^b u_2)
(\partial^2_{\alpha\beta}Z^c u_1)
\right\|_{L^2({\mathbb R}^3)},\\
&
\sum_{|b|+|c|\leq 3}
\left\|
H_{bc}^{\alpha\beta}
(\partial_\alpha Z^b u_1)
(\partial_\beta Z^c u_1)
\right\|_{L^2({\mathbb R}^3)},\\
&
\sum_{|b|+|c|\leq 3}
\left\|
{\tilde H}_{bc}^{\alpha\beta}
(\partial_\alpha Z^b u_1)
(\partial_\beta Z^c u_2)
\right\|_{L^2({\mathbb R}^3)},\\
&
\sum_{|b|+|c|\leq 3}
\left\|
{\hat H}_{bc}^{\alpha\beta}
(\partial_\alpha Z^b u_2)
(\partial_\beta Z^c u_2)
\right\|_{L^2({\mathbb R}^3)},
\end{align}
where the coefficients 
$\{G_{bc}^{\alpha\beta\gamma}\},\dots, \{{\hat H}_{bc}^{\alpha\beta}\}$ 
satisfy the null condition. 
We may focus on (3.50), (3.53), 
because we can treat the others in a similar (and a little easier) way. 

We separate ${\mathbb R}^3$ 
into the two pieces 
$B_t:=\{x\in{\mathbb R}^3:|x|<(t/2)+1\}$ 
and its complement $B_t^c$, 
when handling the $L^2$-norm 
of (3.50), (3.53). 
Let us first consider the $L^2$-norm over $B_t$, 
where we exploit an improved decay rate of solutions. 
For $|b|\leq 1$ and $|c|=2$ 
or $|b|=3$ and $|c|=0$, 
we get by (2.13)
\begin{equation}
\|
(\partial_\gamma Z^b u_2)
(\partial^2_{\alpha\beta}Z^c u_1)
\|_{L^2(B_t)}
\leq
C(1+t)^{-(3/2)+\delta}
\bigl((1+t)^{-\delta}W_4(u_2)\bigr)
W_4(u_1).
\end{equation}
Moreover, for $|b|\leq 2$ and $|c|\leq 1$, 
we get by the H\"older inequality and (2.14) with $p=4$
\begin{align}
\|
(\partial_\gamma Z^b u_2)
(\partial^2_{\alpha\beta}Z^c u_1)
\|_{L^2(B_t)}
&
\leq
\|\partial_\gamma Z^b u_2\|_{L^4(B_t)}
\|\partial_{\alpha\beta}^2 Z^c u_1\|_{L^4(B_t)}\\
&
\leq
C(1+t)^{-(3/2)+\delta}
\bigl((1+t)^{-\delta}W_4(u_2)\bigr)
W_4(u_1).\nonumber
\end{align}
Similarly, 
for $(|b|,|c|)=(3,0), (0,3)$ we get by (2.13)
\begin{equation}
\|
(\partial_\alpha Z^b u_2)
(\partial_\beta Z^c u_2)
\|_{L^2(B_t)}
\leq
C(1+t)^{-(3/2)+2\delta}
\bigl((1+t)^{-\delta}W_4(u_2)\bigr)^2.
\end{equation}
Also, for $|b|\leq 1$ and $|c|\leq 2$ 
or for $|b|\leq 2$ and $|c|\leq 1$, 
we get
\begin{align}
\|
(\partial_\alpha Z^b u_2)
(\partial_\beta Z^c u_2)
\|_{L^2(B_t)}
&
\leq
\|\partial_\alpha Z^b u_2\|_{L^4(B_t)}
\|\partial_\beta Z^c u_2\|_{L^4(B_t)}\\
&
\leq
C(1+t)^{-(3/2)+2\delta}
\bigl((1+t)^{-\delta}W_4(u_2)\bigr)^2.\nonumber
\end{align}
Next, let us consider the $L^2$-norm over $B_t^c$. 
Lemmas 2.2 and 2.3 play an important role here. 
For $|b|\leq 1$ and $|c|=2$, we get by (2.9)
\begin{align}
&
\left\|
{\tilde G}_{bc}^{\alpha\beta\gamma}
(\partial_\gamma Z^b u_2)
(\partial^2_{\alpha\beta}Z^c u_1)
\right\|_{L^2(B_t^c)}\\
&
\hspace{0.2cm}
\leq
C(1+t)^{-1}
\left(
\sum_{|b|\leq 1}
\|ZZ^b u_2\|_{L^\infty(B_t^c)}
+
\sum_{|b|\leq 1}
\|\partial Z^b u_2\|_{L^\infty(B_t^c)}
\right)
W_4(u_1)\nonumber\\
&
\hspace{0.2cm}
\leq
C(1+t)^{-(3/2)+\delta}
\bigl((1+t)^{-\delta}W_4(u_2)\bigr)
W_4(u_1),\nonumber
\end{align}
where we have used (2.15) 
with $p=4$, $s=1$ 
together with 
the Sobolev embedding 
$W^{1,4}(S^2)\hookrightarrow L^\infty(S^2)$. 
Moreover, for $|b|\leq 2$ and $|c|\leq 1$, 
we get by 
$(2.14)$ with $p=4$ 
and 
$(2.15)$ with $p=4$, $s=1$
\begin{align}
&
\left\|
{\tilde G}_{bc}^{\alpha\beta\gamma}
(\partial_\gamma Z^b u_2)
(\partial^2_{\alpha\beta}Z^c u_1)
\right\|_{L^2(B_t^c)}\\
&
\hspace{0.2cm}
\leq
C
(1+t)^{-1}
\bigl(
\|
ZZ^b u_2
\|_{L^{\infty}_r L^4_\omega(B_t^c)}
\|
\partial^2 Z^c u_1
\|_{L^2_r L^4_\omega(B_t^c)}\nonumber\\
&
\hspace{2.7cm}
+
\|
\partial Z^b u_2
\|_{L^4(B_t^c)}
\|
\partial ZZ^c u_1
\|_{L^4(B_t^c)}
\bigr)
\nonumber\\
&
\hspace{0.2cm}
\leq
C(1+t)^{-(3/2)+\delta}
\bigl((1+t)^{-\delta}W_4(u_2)\bigr)
W_4(u_1).\nonumber
\end{align}
Here, we have used the notation
%%%%%%%%%%%%%%%%%%%%%%%%%%%%%%%
\begin{align*}
&
\|w\|_{L_r^\infty L_\omega^4(B_t^c)}
:=
\sup_{r>(t/2)+1}
\|w(r\cdot)\|_{L^4(S^2)},\\
&
\|w\|_{L_r^2 L_\omega^4(B_t^c)}
:=
\biggl(
\int_{(t/2)+1}^\infty \|w(r\cdot)\|_{L^4(S^2)}^2 r^2dr
\biggr)^{1/2}.\nonumber
\end{align*}
%%%%%%%%%%%%%%%%%%%%%%%%%%%%%%%
On the other hand, 
for $|b|=3$ and $|c|=0$ 
we employ (2.5) and (2.13) to get
\begin{align}
&
\left\|
{\tilde G}_{bc}^{\alpha\beta\gamma}
(\partial_\gamma Z^b u_2)
(\partial^2_{\alpha\beta} u_1)
\right\|_{L^2(B_t^c)}\\
&
\hspace{0.2cm}
\leq
C(1+t)^{-1+\eta}
\|
\langle
t-r
\rangle^{-(1/2)-\eta}
TZ^b u_2
\|_{L^2({\mathbb R}^3)}
W_4(u_1)
+
CW_4(u_2)
\|
T\partial u_1
\|_{L^\infty({\mathbb R}^3)}.\nonumber
\end{align}
Similarly, for $(|b|,|c|)=(3,0), (0,3)$, 
we employ (2.6) to get
\begin{align}
&
\left\|
{\hat H}_{bc}^{\alpha\beta}
(\partial_\alpha Z^b u_2)
(\partial_\beta Z^c u_2)
\right\|_{L^2(B_t^c)}\\
&
\hspace{0.2cm}
\leq
C(1+t)^{-1+\eta+\delta}
\left(
\sum_{|d|=3}
\|
\langle
t-r
\rangle^{-(1/2)-\eta}
TZ^d u_2
\|_{L^2({\mathbb R}^3)}
\right)
\bigl(
(1+t)^{-\delta}
W_4(u_2)
\bigr)\nonumber\\
&
\hspace{0.56cm}
+
C
\|
Tu_2(t)
\|_{L^\infty(B_t^c)}
W_4(u_2).\nonumber
\end{align}
For $|b|\leq 2$ and $|c|\leq 1$ 
or 
$|b|\leq 1$ and $|c|\leq 2$, 
we use (2.10) to obtain
\begin{equation}
\left\|
{\hat H}_{bc}^{\alpha\beta}
(\partial_\alpha Z^b u_2)
(\partial_\beta Z^c u_2)
\right\|_{L^2(B_t^c)}
\leq
C(1+t)^{-(3/2)+2\delta}
\bigl(
(1+t)^{-\delta}
W_4(u_2)
\bigr)^2
\end{equation}
as in (3.59). 

Since $\delta$ and $\eta$ are sufficiently small, 
we have only to explain how to 
handle the integral from $0$ to $T$ of 
$\|T\partial u_1(t)\|_{L^\infty({\mathbb R}^3)}W_4(u_2(t))$, 
$\|T u_2(t)\|_{L^\infty(B_t^c)}W_4(u_2(t))$, and 
\begin{equation}
(1+t)^{-1+\eta+\delta}
\sum_{|d|=3}
\|
\langle
t-r
\rangle^{-(1/2)-\eta}
TZ^d u_2
\|_{L^2({\mathbb R}^3)}
\end{equation}
(see (3.60), (3.61)). 
Using (2.12) first and then (2.13), 
we get
\begin{equation}
\|
T\partial u_1(t)
\|_{L^\infty({\mathbb R}^3)}W_4(u_2(t))
\leq
C
(1+t)^{-2+\delta}
W_4(u_1(t))
\bigl(
(1+t)^{-\delta}
W_4(u_2(t))
\bigr).
\end{equation}
Moreover, 
using (2.12), 
the Sobolev embedding 
$W^{1,4}(S^2)\hookrightarrow L^\infty(S^2)$, 
and (2.15), 
we obtain
\begin{align}
\|T u_2(t)\|_{L^\infty(B_t^c)}W_4(u_2(t))
&
\leq
C(1+t)^{-1}
\sum_{|a|=1}
\|
Z^a u_2(t)
\|_{L^\infty(B_t^c)}
W_4(u_2(t))\\
&
\leq
C(1+t)^{-1}
\sum_{|a|=1}^2
\|
Z^a u_2(t)
\|_{L_r^\infty L_\omega^4 (B_t^c)}
W_4(u_2(t))\nonumber\\
&
\leq
C(1+t)^{-(3/2)+2\delta}
\bigl(
(1+t)^{-\delta}
W_4(u_2(t))
\bigr)^2.\nonumber
\end{align}
The estimates (3.64)--(3.56) are strong enough 
to get the desirable bound. 
When considering the integral of (3.63) from $0$ to $T$, 
we may suppose $T>1$ without loss of generality. 
Using the dyadic decomposition of the interval 
$(0,T)$, 
we get as in page 363 of \cite{Sogge2003}, 
\begin{align}
&
\int_0^T
(1+t)^{-1+\eta+\delta}
\|
\langle
t-r
\rangle^{-(1/2)-\eta}
TZ^d u_2
\|_{L^2({\mathbb R}^3)}
dt
\leq
\int_0^1\cdots dt
+
\sum_{j=0}^{N(T)}
\int_{2^j}^{2^{j+1}}
\cdots
dt\\
&
\hspace{0.2cm}
\leq
C\sup_{0<t<1}
\|
\partial Z^d u_2
\|_{L^2({\mathbb R}^3)}\nonumber\\
&
\hspace{0.56cm}
+
\sum_{j=0}^{N(T)}
\left(
\int_{2^j}^{2^{j+1}}
(1+t)^{-2+2(\eta+\delta)}dt
\right)^{1/2}
\|
\langle
t-r
\rangle^{-(1/2)-\eta}
TZ^d u_2
\|_{L^2((2^{j},2^{j+1})\times{\mathbb R}^3)}
\nonumber\\
&
\hspace{0.2cm}
\leq
C\sup_{0<t<T}
(1+t)^{-\delta}
\|
\partial Z^d u_2
\|_{L^2({\mathbb R}^3)}\nonumber\\
&
\hspace{0.56cm}
+
C
\left(
\sum_{j=0}^\infty
(2^j)^{-(1/2)+\eta+2\delta}
\right)
\sup_{0\leq j\leq N(T)}
(2^j)^{-\delta}
\|
\langle
t-r
\rangle^{-(1/2)-\eta}
TZ^d u_2
\|_{L^2((0,2^{j+1})\times{\mathbb R}^3)}\nonumber\\
&
\hspace{0.2cm}
\leq
C
\sup_{0<t<T}
(1+t)^{-\delta}
W_4(u_2(t))\nonumber\\
&
\hspace{0.56cm}
+
C
\sup_{0<t<T}
(1+t)^{-\delta}
\|
\langle
\tau-r
\rangle^{-(1/2)-\eta}
TZ^d u_2
\|_{L^2((0,t)\times{\mathbb R}^3)}.\nonumber
\end{align}
Here, $N(T)$ stands for the natural number 
such that 
$2^{N(T)}<T\leq 2^{N(T)+1}$, 
and in (3.66) we have abused the notation to mean 
$T$ by $2^{N(T)+1}$. 
We have also used the fact that the series above converges 
because $\delta$ and $\eta$ are small so that 
$-(1/2)+\eta+2\delta<0$. 

Taking the assumption (3.46) into account, 
we have finally shown
\begin{align}
&
\sum_{{|b|+|c|\leq 3}\atop{|c|\leq 2}}
\int_0^T
\left\|
\sum_{\alpha,\beta,\gamma=0}^3
{\tilde G}_{bc}^{\alpha\beta\gamma}
(\partial_\gamma Z^b u_2)
(\partial^2_{\alpha\beta}Z^c u_1)
\right\|_{L^2({\mathbb R}^3)}
dt\\
&
\hspace{0.56cm}
+
\sum_{|b|+|c|\leq 3}
\int_0^T
\left\|
\sum_{\alpha,\beta=0}^3
{\hat H}_{bc}^{\alpha\beta}
(\partial_\alpha Z^b u_2)
(\partial_\beta Z^c u_2)
\right\|_{L^2({\mathbb R}^3)}
dt\nonumber\\
&
\hspace{0.2cm}
\leq
C
\bigl(
W_4(u_1(0))
+
W_4(u_2(0))
\bigr)^2,\nonumber
\end{align}
as desired. 
The integral from $0$ to $T$ 
of the $L^2({\mathbb R}^3)$-norm 
in (3.49), (3.51), and (3.52) 
has the same bound. 
We have finished the energy estimate for $u_1$. 

Let us turn our attention to the energy estimate for $u_2$. 
Taking Proposition 3.2 into account and 
recalling the argument above for the energy estimate for $u_1$, 
we may focus on 
\begin{align}
&
\sum_{{|b|+|c|\leq 3}\atop{|c|\leq 2}}
\left\|
G_{bc}^{\alpha\beta\gamma}
(\partial_\gamma Z^b u_1)
(\partial^2_{\alpha\beta}Z^c u_2)
\right\|_{L^2({\mathbb R}^3)},\\
&
\sum_{|b|+|c|\leq 3}
\left\|
H_{bc}^{\alpha\beta}
(\partial_\alpha Z^b u_1)
(\partial_\beta Z^c u_2)
\right\|_{L^2({\mathbb R}^3)},\\
&
\sum_{|b|+|c|\leq 3}
\left\|
{\tilde H}_{bc}^{\alpha\beta}
(\partial_\alpha Z^b u_1)
(\partial_\beta Z^c u_1)
\right\|_{L^2({\mathbb R}^3)},
\end{align}
where the coefficients 
$\{G_{bc}^{\alpha\beta\gamma}\}$, 
$\{H_{bc}^{\alpha\beta}\}$, 
$\{{\tilde H}_{bc}^{\alpha\beta}\}$ 
do not necessarily satisfy the null condition. 
Here we safely omit the explanation of 
how to handle such terms as 
${\tilde G}^{\alpha\beta\gamma}_{bc}
(\partial_\gamma Z^b u_2)
(\partial_{\alpha\beta}^2 Z^c u_2)$, 
${\hat H}^{\alpha\beta}_{bc}
(\partial_\alpha Z^b u_2)
(\partial_\beta Z^c u_2)$ 
with $\{{\tilde G}^{\alpha\beta\gamma}_{bc}\}$, 
$\{{\hat H}^{\alpha\beta}_{bc}\}$ satisfying the null condition, 
because we have already done with the estimate for them 
in the course of the energy estimate for $u_1$ above. 
Using (2.13)--(2.14), 
we get in the standard way
\begin{align}
&
(1+t)^{-2\delta}
\int_0^t
(1+\tau)^\delta
\sum_{{|b|+|c|\leq 3}\atop{|c|\leq 2}}
\left\|
G_{bc}^{\alpha\beta\gamma}
(\partial_\gamma Z^b u_1)
(\partial^2_{\alpha\beta}Z^c u_2)
\right\|_{L^2({\mathbb R}^3)}
d\tau\\
&
\hspace{0.2cm}
\leq
C(1+t)^{-2\delta}
\int_0^t(1+\tau)^{-1+2\delta}d\tau
\left(
\sup_{0<t<T}
W_4(u_1(t))
\right)
\left(
\sup_{0<t<T}
(1+t)^{-\delta}
W_4(u_2(t))
\right)\nonumber\\
&
\hspace{0.2cm}
\leq
C
\bigl(
W_4(u_1(0))
+
W_4(u_2(0))
\bigr)^2.\nonumber
\end{align}
We can deal with (3.69), (3.70) 
in the same way as we have just done in (3.71). 
We have finished the energy estimate for $u_2$. 
The proof of Proposition 3.3 has been finished. 
\end{proof}
Now we are ready to complete the proof of 
the key a priori estimate (3.6). 
Thanks to the size condition (3.1), 
we see by (3.47) that 
there holds for all $T\in (0,T_*)$
\begin{align}
&
\sup_{0<t<T}
A(u(t))
+
\sum_{|a|\leq 3}
\biggl(
\int_0^T\!\!\!\int_{{\mathbb R}^3}
\langle t-r\rangle^{-1-2\eta}
\sum_{i=1}^3
|T_i Z^a u_1(t,x)|^2dtdx
\biggr)^{1/2}\\
&
\hspace{0.56cm}
+
\sum_{|a|\leq 3}
\sup_{0<t<T}
(1+t)^{-\delta}
\biggl(
\int_0^t\!\!\!\int_{{\mathbb R}^3}
\langle \tau-r\rangle^{-1-2\eta}
\sum_{i=1}^3
|T_i Z^a u_2(\tau,x)|^2d\tau dx
\biggr)^{1/2}\nonumber\\
&
\hspace{0.2cm}
\leq
\frac{C_0}{2}
\bigl(
W_4(u_1(0))
+
W_4(u_2(0))
\bigr)
+
\frac{C_0}{6}
\bigl(
W_4(u_1(0))+W_4(u_2(0))
\bigr)\nonumber\\
&
\hspace{0.2cm}
=\frac{2}{3}C_0
\bigl(
W_4(u_1(0))
+
W_4(u_2(0))
\bigr),\nonumber
\end{align}
which completes the proof of (3.6).

We are in a position to show that 
the local solution exists globally in time. 
We first remark that the equality $T_*=T^*$ holds. 
(The definition of $T_*$ and $T^*$ 
is given below (3.2) and (3.5), respectively.) 
Indeed, if we suppose $T_*<T^*$, 
then we see by (3.6), (3.4) and the absolute continuity of integral 
that 
there exists $T'\in (T_*,T^*)$ such that 
\begin{align}
&
\sup_{0<t<T'}A(u(t))
+
\sum_{|a|\leq 3}
\biggl(
\int_0^{T'}\!\!\!\int_{{\mathbb R}^3}
\langle t-r\rangle^{-1-2\eta}
\sum_{i=1}^3
|T_i Z^a u_1(t,x)|^2dtdx
\biggr)^{1/2}\\
&
\hspace{0.56cm}
+
\sum_{|a|\leq 3}
\sup_{0<t<T'}
(1+t)^{-\delta}
\biggl(
\int_0^t\!\!\!\int_{{\mathbb R}^3}
\langle \tau-r\rangle^{-1-2\eta}
\sum_{i=1}^3
|T_i Z^a u_2(\tau,x)|^2d\tau dx
\biggr)^{1/2}\nonumber\\
&
\hspace{0.2cm}
\leq
C_0
\bigl(
W_4(u_1(0))+W_4(u_2(0))
\bigr)
\nonumber
\end{align}
holds, 
which contradicts the definition of $T_*$. 

We now suppose $T^*<\infty$. 
Because of $T^*=T_*$, 
we know among others that 
the inequality 
$A(u(t))
\leq 
C_0
\bigl(
W_4(u_1(0))+W_4(u_2(0))
\bigr)
$ 
holds 
for all $t\in (0,T^*)$. 
If we solve the system (1.10) 
by giving the smooth, compactly supported initial data 
$(u(T^*-\delta,x),\partial_t u(T^*-\delta,x))$ 
at $t=T^*-\delta$ (by $\delta$ we mean a sufficiently small constant), 
then by the standard local existence theorem 
we know that there exists ${\hat T}>T^*$ such that 
the local solution can be continued to a larger strip 
$(0,{\hat T})\times {\mathbb R}^3$. 
Since the solution $u$ is smooth and compactly supported for any fixed time 
$t\in (0,{\hat T})$, 
we know $A(u(t))\in C([0,{\hat T}))$ 
and thus there exists 
${\tilde T}\in (T^*,{\hat T})$ such that 
the local solution satisfies 
$W_4(u_1(t))+(1+t)^{-\delta}W_4(u_2(t))\leq 2C_0$ 
for any $t\leq{\tilde T}$, 
which contradicts the definition of $T^*$. 
We therefore see $T^*=\infty$, and the local solution actually exists 
globally in time. 
This global solution obviously satisfies the estimate (1.13). 
The proof of Theorem 1.4 has been finished.
%%%%%%%%%%%%%%%%%%%%%%%
\section{Proof of Theorem 1.5}
Obviously, it suffices to show (1.18). 
In view of Proposition 3.1 and (3.28)--(3.53), 
we need to deal with
\begin{equation}
\sum_{{|b|+|c|\leq 4}\atop{|c|\leq 3}}
\left\|
G_{bc}^{\alpha\beta\gamma}
(\partial_\gamma Z^b u)
(\partial^2_{\alpha\beta}Z^c u)
\right\|_{L^2({\mathbb R}^3)},
\quad
\sum_{|b|+|c|\leq 4}
\left\|
H_{bc}^{\alpha\beta}
(\partial_\alpha Z^b u)
(\partial_\beta Z^c u)
\right\|_{L^2({\mathbb R}^3)},
\end{equation}
where the coefficients 
$\{G_{bc}^{\alpha\beta\gamma}\}$ 
and 
$\{H_{bc}^{\alpha\beta}\}$ 
satisfy the null condition. 
We again separate 
${\mathbb R}^3$ into the two pieces 
$B_t=\{x\in{\mathbb R}^3:|x|<(t/2)+1\}$ and $B_t^c$. 
We get by (2.13)--(2.14) for any $\alpha$, $\beta$, and $\gamma$ 
\begin{align}
&
\sum_{{|b|+|c|\leq 4}\atop{|c|\leq 3}}
\|
(\partial_\gamma Z^b u)
(\partial^2_{\alpha\beta}Z^c u)
\|_{L^2(B_t)}
+
\sum_{|b|+|c|\leq 4}
\|
(\partial_\alpha Z^b u)
(\partial_\beta Z^c u)
\|_{L^2(B_t)}\\
&
\hspace{0.2cm}
\leq
C(1+t)^{-3/2}
W_4(u(t))W_5(u(t))
\leq
C(1+t)^{-3/2}
W_4(u(0))W_5(u(t)).\nonumber
\end{align}
Moreover, 
by repeating essentially the same argument as in 
(3.58)--(3.62), we get
\begin{align}
&
\sum_{{|b|+|c|\leq 4}\atop{|c|\leq 3}}
\left\|
G_{bc}^{\alpha\beta\gamma}
(\partial_\gamma Z^b u)
(\partial^2_{\alpha\beta}Z^c u)
\right\|_{L^2(B_t^c)}
+
\sum_{|b|+|c|\leq 4}
\left\|
H_{bc}^{\alpha\beta}
(\partial_\alpha Z^b u)
(\partial_\beta Z^c u)
\right\|_{L^2(B_t^c)}\\
&
\hspace{0.2cm}
\leq
C(1+t)^{-3/2}
W_4(u(0))W_5(u(t))\nonumber\\
&
\hspace{0.56cm}
+
C\|T\partial u(t)\|_{L^\infty({\mathbb R}^3)}W_5(u(t))
+
C\|Tu(t)\|_{L^\infty(B_t^c)}W_5(u(t))
\nonumber\\
&
\hspace{0.56cm}
+
C(1+t)^{-1+\eta}
W_4(u(0))
\left(
\sum_{|d|=4}
\|
\langle
t-r
\rangle^{-(1/2)-\eta}
TZ^d u
\|_{L^2({\mathbb R}^3)}
\right).\nonumber
\end{align}
By (3.11), (4.2), and (4.3) 
(see also (3.64)--(3.66)), 
we get for all $t>0$
\begin{align}
&
W_5(u(t))
+
\sum_{|a|\leq 4}
\left(
\sum_{i=1}^3
\|
\langle \tau-r\rangle^{-(1/2)-\eta}
T_i Z^a u
\|_{L^2((0,t)\times{\mathbb R}^3)}^2
\right)^{1/2}\\
&
\hspace{0.2cm}
\leq
CW_5(u(0))
+
CW_4(u(0))
\int_0^t
(1+\tau)^{-3/2}W_5(u(\tau))d\tau
\nonumber\\
&
\hspace{0.56cm}
+
C_4
W_4(u(0))
\sum_{|a|=4}
\left(
\sum_{i=1}^3
\|
\langle \tau-r\rangle^{-(1/2)-\eta}
T_i Z^a u
\|_{L^2((0,t)\times{\mathbb R}^3)}^2
\right)^{1/2}\nonumber
\end{align}
for a constant $C_4>0$. 
Suppose $C_4W_4(u(0))\leq 1$, if necessary. 
Then, we obtain (1.18) by the Gronwall inequality. 
The proof of Theorem 1.5 has been completed.
%%%%%%%%%%%%%%%%%%%%%%%
\appendix
\section{Proof of (1.25)}
Let us prove (1.25) in general space dimensions $n\geq 2$. 
In the following, 
by $Z^a$ we denote the $n$-dimensional analogue of 
what we have meant in the preceding sections. 
Also, by $[p]$ we mean the greatest integer not greater than $p$. 
We prove:
\begin{proposition}
Let $n\geq 2$. 
Suppose $\phi(t,x)\in C^\infty((0,\infty)\times{\mathbb R}^n)$ 
and 
${\rm supp}\,\phi(t,\cdot)\subset
\{x\in{\mathbb R}^n:|x|<t+R\}$ 
for some constant $R>0$. 
Then, there exists a constant $C=C(n,R)>0$ 
with $C\to\infty$ as $R\to\infty$ such that 
\begin{equation}
(1+t)^{(n-1)/2}
(1+|t-r|)^{-1/2}
|\phi(t,x)|
\leq
C
\sum_{|a|\leq [(n-1)/2]+1}
\|
\partial_x
Z^a
\phi(t,\cdot)
\|_{L^2({\mathbb R}^n)}
\end{equation}
holds. 
\end{proposition}
\begin{proof}
We consider the two cases 
$|x|>(t/2)+1$ and $|x|<(t/2)+1$ separately. 
For the former, we follow the outline made 
on page 610 of \cite{Al2001}. 
Since 
$$
\phi(t,x)
=
\phi(t,r\omega)
-
\phi(t,(t+R)\omega)
=
\int_{t+R}^r
\frac{\partial}{\partial\rho}
\phi(t,\rho\omega)d\rho,
$$
we get by the H\"older inequality
\begin{align*}
|\phi(t,x)|
&
\leq
\int_r^{t+R}
|(\nabla\phi)(t,\rho\omega)|d\rho\\
&
\leq
(t-r+R)^{1/2}
\left(
\int_r^{t+R}
|\nabla\phi(t,\rho\omega)|^2
\rho^{n-1}d\rho
\right)^{1/2}
r^{-(n-1)/2}\nonumber
\end{align*}
whenever $r<t+R$. 
Combined with the Sobolev embedding on $S^{n-1}$, 
the last inequality yields
\begin{align}
&
r^{(n-1)/2}
(t-r+R)^{-1/2}
|\phi(t,x)|\\
&
\hspace{0.2cm}
\leq
C
\sum_{|a|\leq [(n-1)/2]+1}
\left(
\int_0^{t+R}\!\!\!\int_{S^{n-1}}
|(\nabla\Omega^a \phi)(t,\rho\omega)|^2
\rho^{n-1}d\rho dS
\right)^{1/2},\nonumber
\end{align}
which yields (A.1) for $|x|>(t/2)+1$. 

For the latter case $|x|<(t/2)+1$, 
we first note that by the Sobolev embedding 
$H^{[n/2]+1}(\Omega_1)
\hookrightarrow 
L^\infty(\Omega_1)$ 
and the standard scaling argument, 
the inequality
\begin{equation}
\|
\phi(t,\cdot)
\|_{L^\infty(\Omega_\lambda)}
\leq
C\lambda^{-n/2}
\sum_{|a|\leq [n/2]+1}
\lambda^{|a|}
\|
(\partial_x^a\phi)(t,\cdot)
\|_{L^2(\Omega_\lambda)}
\end{equation}
holds for all $\lambda>0$. 
Here and in the following, 
we use the notation 
$\Omega_\lambda
:=
\{x\in{\mathbb R}^n:|x|<\lambda\}$. 
Recalling the notation 
$B_t=\{x\in {\mathbb R}^n:|x|<(t/2)+1\}$ and 
using (3.27) together with the fact that 
for any $k\in{\mathbb N}$, 
$\Omega_{ij}\{(t^2-r^2)^{-k}\}
=
L_j\{(t^2-r^2)^{-k}\}=0$ 
and 
$S\{(t^2-r^2)^{-k}\}=(-2k)\{(t^2-r^2)^{-k}\}$, 
we get from (A.3) with $\lambda=(t/2)+1$
\begin{align}
\|
\phi(t,\cdot)
\|_{L^\infty(B_t)}
&
\leq
C(1+t)^{-n/2}
\sum_{|a|\leq [n/2]+1}
\|
Z^a \phi(t,\cdot)
\|_{L^2(B_t)}\\
&
\leq
C(1+t)^{-n/2}(t+R)
\sum_{|a|\leq [n/2]+1}
\left\|
\frac{1}{t-|\cdot|+R}
Z^a\phi(t,\cdot)
\right\|_{L^2(B_t)}\nonumber\\
&
\leq
C_R(1+t)^{-(n/2)+1}
\sum_{|a|\leq [n/2]+1}
\left\|
\frac{1}{t-|\cdot|+R}
Z^a\phi(t,\cdot)
\right\|_{L^2(\Omega_{t+R})}\nonumber\\
&
\leq
C_R(1+t)^{-(n/2)+1}
\sum_{|a|\leq [n/2]+1}
\|
\partial_x Z^a\phi(t,\cdot)
\|_{L^2({\mathbb R}^n)}.\nonumber
\end{align}
Here, we have used the well-known inequality of Lindblad \cite{Lind}. 
We see that by (A.4) that 
(A.1) holds also for $|x|<(n/2)+1$. 
The proof has been finished.
\end{proof}
%%%%%%%%%%%%%%%%%%%%%%
%\section*{Acknowledgement}
%%%%%%%%%%%%%%%%%%%%%%%

 \bigskip
%%%%%%%%%%%% 著者所属 %%%%%%%%%%%%%
 \address{% 第一著者
Department of Mathematics\\
Faculty of Education\\
Mie University\\
1577 Kurima-machiya-cho,Tsu\\
Mie Prefecture 514-8507 JAPAN
}
 {hidano@edu.mie-u.ac.jp}% 第一著者の後は改行せずに続ける
%%%%%%%%%
\address{% 第二著者
Hokkaido University of Science\\
7-Jo 15-4-1 Maeda, Teine, Sapporo\\
Hokkaido 006-8585 JAPAN
}
 {yokoyama@hus.ac.jp}
\end{document}